\documentclass[11pt]{amsart}
\usepackage{amsmath,amsfonts,amssymb,amsthm,epsfig,color, mathrsfs, scalerel, stackengine} 
\usepackage{hyperref} 
\usepackage{bbm}
\usepackage{mathrsfs, euscript}
\usepackage{upgreek}

\newcommand{\bb}{\mathbb}

\newcommand{\CC}{\bb C}
\newcommand{\h}{\bb H}
\newcommand{\HH}{\bb H}
\newcommand{\Z}{\bb Z}
\newcommand{\ZZ}{\bb Z}
\newcommand{\R}{\bb R}
\newcommand{\RR}{\bb R}

\newcommand{\NN}{\bb N}

\newcommand{\Ga}{\Gamma}

\newcommand{\wrt}[1]{\mathrm{d}{#1}}

\newcommand{\supp}{\operatorname{supp}}

\newcommand{\SL}{\operatorname{SL}}

\newcommand{\PSL}{\operatorname{PSL}}

\def\ind{\mathbbm{1}}
\newtheorem{Theorem}{Theorem}
\numberwithin{Theorem}{section}

\newtheorem{coro}[Theorem]{Corollary}

\newtheorem{theo}[Theorem]{Theorem}
\newtheorem{prop}[Theorem]{Proposition}

\newtheorem{lemm}[Theorem]{Lemma}

\newcommand\reallywidehat[1]{%
\savestack{\tmpbox}{\stretchto{%
  \scaleto{%
    \scalerel*[\widthof{\ensuremath{#1}}]{\kern-.6pt\bigwedge\kern-.6pt}%
    {\rule[-\textheight/2]{1ex}{\textheight}}
  }{\textheight}%
}{0.5ex}}%
\stackon[1pt]{#1}{\tmpbox}%
}

\newtheorem*{lemma*}{Lemma}
\newtheorem*{question*}{Question}

\newtheorem*{theorem*}{Theorem}

\theoremstyle{remark}

\newtheorem{rema}[Theorem]{\sc Remark}
\newtheorem*{rema*}{\sc Remark}
\newtheorem{example}[Theorem]{\bf Example}
\numberwithin{equation}{section}
\begin{document}
\title{Shrinking target equidistribution of horocycles in cusps}
\author{Jimmy Tseng}

\address{Department of Mathematics, University of Exeter, Exeter EX4 4QF, U.K}
\email{j.tseng@exeter.ac.uk}

 \thanks{The author was supported by EPSRC grant EP/T005130/1.}

\begin{abstract}  Consider a shrinking neighborhood of a cusp of the unit tangent bundle of a noncompact hyperbolic surface of finite area, and let the neighborhood shrink into the cusp at a rate of $T^{-1}$ as $T \rightarrow \infty$.  We show that a closed horocycle whose length $\ell$ goes to infinity or even a segment of that horocycle becomes equidistributed on the shrinking neighborhood when normalized by the rate $T^{-1}$ provided that $T/\ell \rightarrow 0$ and, for any $\delta>0$, the segment remains larger than $\max\left\{T^{-1/6},\left(T/\ell\right)^{1/2}\right\}\left(T/\ell\right)^{-\delta}$.  We also have an effective result for a smaller range of rates of growth of $T$ and $\ell$.  Finally, a number-theoretic identity involving the Euler totient function follows from our technique.

\end{abstract}

\maketitle

\tableofcontents

\section {Introduction}\label{secBackground}

 Let $G:= \PSL_2(\R)$, $\Ga \subset G$ be a cofinite Fuchsian group, $\kappa_1, \cdots, \kappa_q \subset \RR \cup \{\infty\}$ be inequivalent cusps with stabilizer groups $\Gamma_1, \cdots, \Gamma_q$ where $q \in \NN$, $\h$ be the upper-half plane model of the hyperbolic plane, and $T^1\h$ be its unit tangent bundle.  We may assume, by conjugation of $\Gamma$, that $\kappa_1=\infty$ and has stabilizer group  \[\Ga_1:=\Ga_\infty := \left\{ \left( \begin{array}{cc}
1 & b \\
0 & 1 \end{array} \right) \bigg{\vert} \ b \in \Z \right\}.\]  The cusp $\kappa_1$ is called the {\em standard cusp.}  The group $G$ acts on $\h$ and $T^1\h$ as follows:  for any $\gamma:=\begin{pmatrix}a & b \\ c& d  \end{pmatrix} \in G$, we have \begin{align}\label{eqn:MobiusAct} \gamma(z) &= \left(\frac{az+b}{cz+d}\right) \quad \quad \gamma(z, \theta) = \left(\frac{az+b}{cz+d}, \theta - 2 \arg(cz+d)\right),\end{align} where $\theta$ is an angular variable measured from the upward vertical counterclockwise.  Let $z:=x+iy$.  We may identify $G$ with $T^1\h$ via the mapping $\gamma \mapsto \gamma(i, 0)$ and also identify $\Gamma \backslash G$ with $\Gamma \backslash T^1 \h$.  The Haar measure $\mu$ on $G$ (which is unique up to a multiplicative constant) is identified with the Liouville volume measure $y^{-2}\wrt x~\wrt y~\wrt \theta$ on $T^1 \h$.  Let $T^1\left(\Gamma \backslash  \h\right)$ be the unit tangent bundle of the surface $\Gamma \backslash \h$.  Then we have $T^1\left(\Gamma \backslash  \h\right)=\Gamma \backslash T^1 \h$ except at elliptic fixed points (where $\Gamma \backslash \h$ is an orbifold).

Let $\alpha < \beta$ be real numbers and consider the horocycle $\{(x+iy, 0): \alpha \leq x \leq \beta \}$.  As $y \rightarrow 0$, this horocycle {\em equidistributes,} namely  \[\frac 1{\beta -\alpha} \int_\alpha^\beta f(x+iy, 0)~\wrt x \rightarrow \frac 1{\mu(\Gamma \backslash G)} \int_{\Gamma \backslash G} f(z, \theta)~\wrt \mu\] for any $f \in C_c\left(\Gamma \backslash G\right)$.  Such equidistribution results have been studied by many mathematicians such as Zagier~\cite{Zag79}, Sarnak~\cite{Sa}, Dani and Smillie~\cite{DS84}, Hejhal~\cite{Hej00, Hej95}, Flaminio and Forni~\cite{FF03}, and Str\"ombergsson~\cite{Str13, St}.  A natural question, asked by Hejhal, is the question of what happens if $\beta-\alpha$ also shrinks.  Clearly, if $\beta-\alpha$ shrinks too quickly relative to $y$, then the horocycle can not equidistribute.  If, however, $\beta-\alpha \geq y^{C(\Gamma)-\varepsilon}$ holds, then the horocycle equidistributes uniformly, as  Hejhal~\cite{Hej00} showed for the exponent $C(\Gamma)=1/3$ and Str\"ombergsson~\cite{St} showed for the best possible exponent $C(\Gamma)=1/2$.  Some of these results are even {\em effective}, namely an error rate is computed.

In this paper, we study another natural and related question, namely the question of what happens if we take a family of test functions with supports shrinking into one of the cusps at a rate of, say, $T^{-1}$ for $T \rightarrow \infty$.  If $y$ shrinks fast enough and $\beta-\alpha$ does not shrink too fast, then we would expect the horocycle to equidistribute uniformly, provided that we divide by the rate $T^{-1}$:  \begin{align}\label{eqn:UnifNormHoroEquid}\frac T{\beta -\alpha} \int_\alpha^\beta \cdot ~\wrt x \rightarrow \frac 1{\mu(\Gamma \backslash G)} \int_{\Gamma \backslash G}\cdot ~\wrt \mu \quad \quad \textrm{ as $T \rightarrow \infty$ and $y \rightarrow 0$.}\end{align}  Let us refer to (\ref{eqn:UnifNormHoroEquid}) as {\em shrinking target horocycle equidistribution (STHE)} for such a family of test functions.  Shrinking target horocycle equidistribution is what we show in this paper (see Theorem~\ref{thm:TanBunRelEqui}) for a large collection of such test functions.  A necessary condition for STHE, for our large collection of test functions, is $Ty \rightarrow 0$.  On the other hand, if $Ty \not \rightarrow 0$, then there are cases for which STHE does not hold (even if $\beta-\alpha$ is fixed).  See Example~\ref{eg:NonEquidWhenTyLarge}. With more constraints on the various rates, we can also show an effective result (see Theorem~\ref{thm:EffectTanBunRelEqui}).  

Finally, we note that {\em shrinking target equidistribution} could more generally be formulated on a finite-volume space with cusps and with a geometric object that, under a flow, equidistributes on that space.  For horospheres on $\SL(d, \ZZ) \backslash \SL(d, \RR)$ where $d \geq 2$, shrinking target equidistribution has been very recently shown in~\cite{Tse22}.  A natural pair of parameters, the {\em critical exponent of relative rate} $c_r$ and the {\em normalizing exponent} $c_e$, are defined, and this pair $(c_r, c_e)$ is shown to be equal to $(d, d-1)$~\cite[Introduction]{Tse22}.  Our results, Theorems~\ref{thm:TanBunRelEqui} and~\ref{thm:EffectTanBunRelEqui}, agree with the results in~\cite{Tse22} for the common case where $\Gamma$ is $\PSL(2, \ZZ)$ and show, more generally, that this pair is equal to $(2,1)$ for every cusp of every cofinite Fuchsian group that we consider in this paper.  It may be interesting from a geometric point of view to understand how the pair $(c_r, c_e)$, when it is defined, behaves for cusps of other finite-volume spaces.

\subsection{Statement of results}  Our two main results, Theorems~\ref{thm:TanBunRelEqui} and~\ref{thm:EffectTanBunRelEqui}, both give shrinking target horocycle equidistribution.  Theorem~\ref{thm:EffectTanBunRelEqui} is effective while Theorem~\ref{thm:TanBunRelEqui} allows for a larger range of the rate of growth of $T$ versus decay of $y$ (or, more precisely, the decay of $Ty$).  Here $\phi_{T, \eta}(z, \theta)$ is an automorphic function on $T^1(\h)$ defined by \begin{align}\label{eqn:TestFctDefn}\phi_{T, \eta}(z, \theta) :=\phi_{T, \eta}^{(\kappa_j)}(z, \theta) := \sum_{\gamma \in \Gamma_j \backslash \Gamma} f_{T, \eta}(\sigma_j^{-1}\gamma(z, \theta)), \end{align} where $f_{T,\eta}(z, \theta)$ is a function whose support lies in $[0,1] \times [T-|\eta|, \infty) \times [0, 2\pi)$ and $\sigma_j$ is a scaling matrix.  The function $f_{T, \eta}$ is constructed using another function $h$, and both of these functions, along with $\sigma_j$, $B_0$, and $B_1$, are defined in Section~\ref{subsec:TestFcnOnCusps}. 

\begin{theo}\label{thm:TanBunRelEqui}  Let $j \in \{1, \cdots, q\}$, $\delta>0$, $T > y$, and $0<Ty<1$.  Then we have \begin{align*}\label{eqn:TanBunRelEqui}
 \frac T {\beta - \alpha} \int_\alpha^\beta \phi^{(\kappa_j)}_{T,0}(x+iy, 0)~\wrt x \rightarrow  \frac{1}{\mu(\Gamma \backslash G)} \int_0^{2\pi}\int_0^1 h(x, \theta)~\wrt x~\wrt\theta \end{align*} as $T \rightarrow \infty$ and $Ty \rightarrow 0$, uniformly for \[\beta - \alpha \geq \max\left\{T^{-1/6},\left(\frac{Ty}{B_1}\right)^{1/2}\right\}\left(\frac{Ty}{B_1}\right)^{-\delta}.\]  
 
\end{theo}

\begin{theo}\label{thm:EffectTanBunRelEqui}  Let $0\leq \alpha < \beta\leq 1$, $j \in \{1, \cdots, q\}$, $\frac 1 2 \geq \delta>0$, $|\eta| \leq \min\left(\frac{B_1-B_0}{2}, \frac 1 4 \right)$, $T -\min\left(\frac{B_1-B_0}{2}, \frac 1 4 \right)  >y$, and $0<Ty<\frac 3 4$.  Then, for $\eta \neq 0$, we have \begin{align*}
 \frac T {\beta - \alpha} \int_\alpha^\beta & \phi^{(\kappa_j)}_{T,\eta}(x+iy, 0)~\wrt x \\ &= \begin{cases} \left\langle \phi^{(\kappa_j)}_{T,\eta}\right\rangle+E(\alpha, \beta, T, y, B_1, s_1, s'_1,\eta) + O\left(\frac{Ty}{B_1}\right)^{\delta/2}  & \textrm{ if } \alpha \neq 0 \textrm { or } \beta \neq 1 \\
\left\langle \phi^{(\kappa_j)}_{T,\eta}\right\rangle+E(\alpha, \beta, T, y, B_1, s_1, s'_1,\eta) + O\left(\left(\frac{Ty}{B_1}\right)^{\delta}+\frac 1 {\sqrt{T}}\right)  & \textrm{ if } \alpha =0 \textrm { and } \beta = 1 \end{cases}\end{align*} and, for $\eta =0$ and any $0<|\widetilde{\eta}| \leq \min\left(\frac{B_1-B_0}{2}, \frac 1 4 \right)$, we have   
\begin{align*}
 \frac T {\beta - \alpha} \int_\alpha^\beta & \phi^{(\kappa_j)}_{T,0}(x+iy, 0)~\wrt x \\ &= \begin{cases}  \left\langle \phi^{(\kappa_j)}_{T,0}\right\rangle+E(\alpha, \beta, T, y, B_1, s_1, s'_1,\widetilde{\eta}) + O\left(\left(\frac{Ty}{B_1}\right)^{\delta/2}+ \frac 1 T\right)  & \textrm{ if } \alpha \neq 0 \textrm { or } \beta \neq 1 \\
 \left\langle \phi^{(\kappa_j)}_{T,0}\right\rangle+E(\alpha, \beta, T, y, B_1, s_1, s'_1,\widetilde{\eta}) + O\left(\left(\frac{Ty}{B_1}\right)^{\delta}+\frac 1 {\sqrt{T}}\right)  & \textrm{ if } \alpha =0 \textrm { and } \beta = 1 \end{cases}\end{align*} as $T \rightarrow \infty$ and $Ty \rightarrow 0$, uniformly for \[1 \geq \beta -\alpha \geq \max\left(T^4 \left(\frac{Ty}{B_1}\right)^{1/2}, T^{-1/6}\right)\left(\frac{Ty}{B_1}\right)^{-\delta}.\]

Here, \begin{align*} &\left\langle \phi^{(\kappa_j)}_{T,\eta}\right\rangle := \frac {T}{\mu(\Gamma \backslash G)} \int_0^{2\pi} \int_0^1 \int_0^\infty \ind_{[T, \infty), \eta}(y) h(x, \theta) \frac{\wrt y \wrt x \wrt \theta}{y^2}, \\
&\left\langle \phi^{(\kappa_j)}_{T,0}\right\rangle := \frac{1}{ \mu(\Gamma\backslash G)} \int_0^{2\pi} \int_0^1  h(x, \theta)~\wrt x~\wrt \theta, \\
&E(\alpha, \beta, T, y, B_1, s_1, s'_1,\cdot):=O\left(T^4 |\cdot|^{-4}\right) \\ & \quad \times \left( \left(\frac{Ty}{(\beta-\alpha)^2B_1}\right)^{1/2} \log^2\left(\frac{(\beta-\alpha)B_1}{Ty}\right) + \left(\frac{Ty}{(\beta-\alpha)^2B_1}\right)^{1-s'_1} + \left(\frac{Ty}{(\beta-\alpha)B_1}\right)^{1-s_1}\right),
  \end{align*} $s_1$, $s'_1$ are as in (\ref{eqn:SmallEigenvalues}), and the implied constants (including the one coming from $E$) depend on $\Gamma$, $\kappa_1$, $\kappa_j$, and $h$.
\end{theo}
\begin{rema}
 Note that $Ty \rightarrow 0$ is necessary (in both Theorems~\ref{thm:TanBunRelEqui} and~\ref{thm:EffectTanBunRelEqui}); see Example~\ref{eg:NonEquidWhenTyLarge}.  \end{rema}

\subsection{Outline of the main proofs}  The basic idea is to eliminate the growth of $T$ by renormalizing in the $y$-coordinate to a fixed value $B_1$ (see Section~\ref{sec:RenormShrinkNeigh}) and to apply the usual horocycle equidistribution (see Section~\ref{sec:HorocycleEquiFixedFunc}).  The renormalization is done using the double coset decomposition from Lemma~\ref{lemm:DoubleCosetIntegral}.  To illustrate the renormalization in the simplest case of the closed horocycle (namely for $\alpha=0$ and $\beta=1$), we apply Lemma~\ref{lemm:DoubleCosetIntegral} to $\int_0^1\phi^{(\kappa_j)}_{T,0}(x+iy,0)~\wrt x$ and argue as in the proof of Proposition~\ref{prop:MovingToFixed} to obtain \begin{align*} y  \sum \int_{-\sqrt{B}}^{\sqrt{B}} f_{T,0}\left(\frac a c - \frac 1{c^2 y} \frac 1 {x+i}, - 2\arg(x+i) \right)~\wrt x,
  \end{align*} where the sum is over the double cosets, $B$ is as defined in the statement of Proposition~\ref{prop:MovingToFixed}, and $f_{T,0}$ (which involves the function $h$) is as defined in (\ref{eqnDefnFT0}).  Similarly, for $\int_0^1\phi^{0, (\kappa_j)}_{0}\left(x+i\frac{T}{B_1}y, 0\right)~\wrt x$, we have \begin{align*} \frac{T}{B_1}y  \sum \int_{-\sqrt{B}}^{\sqrt{B}} f^0_0\left(\frac a c - \frac {B_1}{c^2 Ty} \frac 1 {x+i}, - 2\arg(x+i) \right)~\wrt x
  \end{align*} where $f^0_0$ is defined in (\ref{eqn:defnAuxilFunctions}).  Note that, while the support of $\phi^{(\kappa_j)}_{T,0}$ shrinks into cusp $\kappa_j$ as $T \rightarrow \infty$, the support of $\phi^{0, (\kappa_j)}_{0}$ is fixed, and it is to these automorphic functions of fixed support that we apply the usual horocycle equidistribution.  Comparing term-by-term allows us to obtain the estimate \[\left| \frac{T}{B_1}\int_0^1\phi^{(\kappa_j)}_{T,0}(x+iy,0)~\wrt x -  \int_0^1\phi^{0, (\kappa_j)}_{0}\left(x+i\frac{T}{B_1}y, 0\right)~\wrt x\right|\] under suitable conditions, and these and related estimates give STHE.  The key part of the term-by-term comparison (for closed horocycles and in general) is estimating the integrals $\int_{-\sqrt{B}}^{\sqrt{B}} \cdot~\wrt x$, which we accomplish via the method of stationary phase.

Two significant difficulties arise, both involving the $x$-coordinate.  Dealing with general $h(x, \cdot)$, even just for closed horocycles, is a difficulty, which the method of stationary phase allows us to handle via the Fourier expansion of $h(x, \cdot)$.  We show that the non-constant Fourier modes are negligible, and, thus, can replace  $h(x, \cdot)$ with its average.  See Section~\ref{sec:ApplicationMethStatPhase}.  Dealing with general $\beta - \alpha$ is another difficulty, which we handle by taking the Fourier expansion of a suitable smoothing and applying a theorem of Jackson to handle the high Fourier modes.  See Sections~\ref{sec:ProofthmTanBunRelEqui} and~\ref{sec:ProofthmEffectTanBunRelEqui} in which the proofs of Theorems~\ref{thm:TanBunRelEqui} and~\ref{thm:EffectTanBunRelEqui}, respectively, are given.

\subsection*{Acknowledgements}  I wish to thank Jens Marklof for helpful discussions and the referees for their helpful comments.  One of the referees pointed me to~\cite[Theorem~4]{Goo83} for which I am grateful.

For the purpose of open access, the author has applied a CC BY public copyright licence (where permitted by UKRI, ‘Open Government Licence’ or ‘CC BY-ND public copyright licence’ may be stated instead) to any Author Accepted Manuscript version arising.  This study did not generate any new data.

\section{Functions and the double coset decomposition}  In this section, we define our test and auxiliary functions and develop a useful expansion for both of these via the double coset decomposition.

\subsection{Test and auxiliary functions}\label{subsec:TestFcnOnCusps} First, let us describe the space that contains the support of the test and auxiliary functions.  To each cusp $\kappa_j$, there is an element $\sigma_j \in G$ (called a {\em scaling matrix}) such that $\sigma_j \infty = \kappa_j$ and $\sigma_j^{-1} \Gamma_j \sigma_j = \Gamma_\infty$.  We require that $\sigma_1$ be the identity element of $G$.  Let $F \subset [0,1] \times (0, \infty) \subset \h$ be a canonical fundamental domain for the action of $\Gamma$ on $\h$ and $\overline{F}$ be its topological closure (in the Riemann sphere).  By modifying $\sigma_j$ for $j \in \{2, \cdots, q\}$, we can ensure that \begin{align}\label{eqn:FullCuspInY}
 \sigma^{-1}_j(\overline{F}) \cap \{z \in \h : y \geq B\} = [0,1] \times [B, \infty) \end{align} holds for all $j \in \{1, \cdots, q\}$ and for all $B \geq B_0>1$ (see~\cite[(2.2)]{St} or~\cite[Page 268]{He}).  Here $B_0$ is a fixed constant depending on $\Ga$.  Note that $\overline{F} \times  [0, 2\pi)$ is a fundamental domain for the action of $\Gamma$ on $T^1\h$.

For any cusp $\kappa_j$, we will define our test functions in two steps.  The first step is to construct suitable functions with support in $[0,1] \times [B_1, \infty) \times [0, 2\pi)$.  For a set $S \subset (0, \infty)$, define its indictor function $\ind_S:  (0, \infty) \rightarrow [0,1]$ by \[\ind_S(y) := \begin{cases} 
      1 & y \in S\\
      0 & y \in (0,\infty) \backslash S \end{cases}.\]  Let $h: \RR / \ZZ \times  \RR / (2 \pi \ZZ) \rightarrow \CC$ be a $C^\infty$- function and $h_2(x, \theta) := \frac{\partial^2 h(x, \theta)}{\partial x^2}$.  Choose a fixed constant $B_1 >B_0$.  Let $T \geq B_1$ and $\eta$ be a real number for which $\min\left(\frac{B_1-B_0}{2}, \frac 1 4 \right) \geq |\eta|> 0$ holds.  Then define \begin{align}\label{eqnDefnFT0}
f_{T, 0}:= \ind_{[T, \infty)}(y) h(x, \theta), \end{align} which is our main interest in this paper.  But, in order to give an effective result (see Theorem~\ref{thm:EffectTanBunRelEqui}), we will use (and also give results for) approximations by the smooth functions \[f_{T, \eta}: = \ind_{[T, \infty), \eta}(y) h(x, \theta).\]  Here, $\ind_{[T, \infty), \eta}:  (0, \infty) \rightarrow [0,1]$ is an approximation to $\ind_{[T, \infty),0}(y):=\ind_{[T, \infty)}(y)$ defined by the $C^\infty$-function $\ind_{[T, \infty), \eta}(y) :=
      \ind_{[-T+\eta/2, \infty)} * \rho_{|\eta|/2}(y-2T)$ where $*$ denotes convolution and $\rho \in C^\infty_c(\RR)$ is the well-known mollifier defined in (\ref{eqn:MolliferRhoDefn}).  Note that $0\leq \ind_{[T, \infty), \eta} \leq 1$ and \[\supp(\ind_{[T, \infty), \eta}) \subset \begin{cases}  [T - |\eta|, \infty) & \textrm{ if } \eta <0 \\  [T, \infty) & \textrm{ if } \eta >0\end{cases}.\]

The second step is to use (\ref{eqn:FullCuspInY})  to construct, for each $f_{T, \eta}$, the related automorphic function on $T^1(\h)$ denoted by $\phi_{T, \eta}^{(\kappa_j)}(z, \theta)$ and defined in (\ref{eqn:TestFctDefn}).  Here, $\eta$ can also be zero.  Note that, as $f_{T, \eta}$ is supported in the unit tangent bundle of the standard cusp, all but at most one term in the automorphic function $\phi_{T, \eta}^{(\kappa_j)}$ is zero.  Of course, the functions $\phi_{T, \eta}^{(\kappa_j)}$ themselves are supported in the unit tangent bundle of the cusp $\kappa_j$.  The collection of $\phi_{T, \eta}^{(\kappa_j)}$ is our collection of test functions.

Finally, the auxiliary functions, which we will use to study our test functions, are defined as follows.  Let $\mathfrak{m}_w(y) := wy$ for $w \in \CC$.  For $|\eta|\leq \min\left(\frac{B_1-B_0}{2}, \frac 1 4 \right)$, consider \begin{align}\label{eqn:defnAuxilFunctions}  g_{T, \eta}(z, \theta) &:=  \ind_{[T, \infty), \eta}(y) \int_0^1 h(t, \theta)~\wrt t \quad \quad \varphi_{T, \eta}(z, \theta) := \varphi^{(\kappa_j)}_{T, \eta}(z, \theta) :=  \sum_{\gamma \in \Gamma_j \backslash \Gamma} g_{T, \eta}(\sigma_j^{-1}\gamma(z, \theta)), \\\nonumber f_\eta^0(z, \theta) &:=  \ind_{[T, \infty), \eta} \circ \mathfrak{m}_{\frac T {B_1}}(y) \int_0^1 h(t, \theta)~\wrt t \quad \quad  \phi^{0}_\eta(z, \theta) := \phi^{0,(\kappa_j)}_\eta(z, \theta) := \sum_{\gamma \in \Gamma_j \backslash \Gamma} f_\eta^0(\sigma_j^{-1}\gamma(z, \theta)).\end{align}  Note that, since $|\eta|$ is small (or zero), the support of $f_\eta^0$ in the $y$-variable is approximately (or exactly, respectively) $[B_1, \infty)$.

\subsection{The double coset decomposition}  Now consider a cuspidal function in the unit tangent bundle of the standard cusp $f: T^1(\Gamma_\infty \backslash \h) \rightarrow \CC$ such that the support of $f(z,\theta)$ is $[B_0, \infty)$ in the $y$-variable.  We note that $f$ is periodic:  $f(z+1, \theta)=f(z,\theta)=f(z, \theta+2\pi)$.  We can now define a family of automorphic functions on $T^1(\h)$ by \[\phi^{(\kappa_j)}(z, \theta) := \sum_{\gamma \in \Gamma_j \backslash \Gamma} f(\sigma_j^{-1}\gamma(z, \theta)) = \sum_{\gamma \in \Gamma_\infty \backslash \sigma_j^{-1}\Gamma} f(\gamma(z, \theta))\] for each $j \in \{1, \cdots, q\}$.  The double coset decomposition (see~\cite[Section~2.4]{Iwa02} for example) allows us, furthermore, to write \begin{align}\label{eqn:DoubleCosetDecomp}\phi^{(\kappa_j)}(z, \theta) = \delta_{1j} \ f(z, \theta) + \sum_{n \in \ZZ} \  \sum_{\substack{\gamma \in \Gamma_\infty \backslash \sigma_j^{-1}\Gamma/\Gamma_\infty \\ \gamma \neq \omega_\Gamma}} f(\gamma (z+n, \theta))\end{align} where \[\omega_\Gamma:=\Gamma_\infty \begin{pmatrix}  1 & 0 \\ 0 & 1 \end{pmatrix}\Gamma_\infty, \quad \quad \delta_{1j}:=\begin{cases} 1 &\textrm{ if } j=1 \\ 0 &\textrm{ if } j \in \{2, \cdots, q\} \end{cases}.\]  Note that $\gamma = \omega_\Gamma$ can only occur in the case $j=1$.

The following expansion, which applies to both the test and auxiliary functions, is the key setting of our proofs.  For conciseness of notation, we use $e(w) := e^{2 \pi i w}$ for any $w \in \CC$.  
 
\begin{lemm}\label{lemm:DoubleCosetIntegral}  Let $m \in \ZZ$.  We have \begin{align*}
 \int_0^1 \phi^{(\kappa_j)}&(x+iy, 0) \ e(m x)~\wrt x =  \delta_{1j}  \int_0^1  f(x+iy, 0) \ e(m x)~\wrt x   \\ 
 &+  y  \sum_{\substack{\Gamma_\infty \begin{pmatrix}  a & b \\ c & d \end{pmatrix} \Gamma_\infty \in \Gamma_\infty \backslash \sigma_j^{-1}\Gamma/\Gamma_\infty \\  \quad c>0}} \int_{-\infty}^\infty f\left(\frac a c - \frac 1{c^2 y} \frac 1 {x+i}, - 2\arg(x+i) \right) e\left(m xy- m \frac d c\right)~\wrt x.\end{align*}

\end{lemm}

\begin{proof}
 In the double coset decomposition, the condition $\Gamma_\infty \begin{pmatrix}  a & b \\ c & d \end{pmatrix} \Gamma_\infty \neq \omega_{\Gamma}$ is equivalent to the condition $c>0$.  Summing over $n$, we have \begin{align*}
 &\int_{-\infty}^\infty f\left(\frac{a(x+iy)+b}{c(x+iy)+d)}, -2\arg(cx+d+icy) \right) e(mx)~\wrt x  \\ & =   \int_{-\infty}^\infty f\left(\frac a c - \frac 1{c^2 } \frac 1 {x+iy}, - 2\arg(x+iy) \right) e\left(mx - m \frac d c\right)~\wrt x, \end{align*}where we have changed variables $x+d/c \mapsto x$ and simplified.  To obtain the desired result, we again change variables $x/y \mapsto x$.\end{proof}

\section{An application of the method of stationary phase}\label{sec:ApplicationMethStatPhase}  In this section, we state Proposition~\ref{prop:MovingToFixed}, its corollary (Corollary~\ref{coro:MovingToFixed}), and a generalization of the corollary (Theorem~\ref{thm:EscNeighDepOnConst}) and give Example~\ref{eg:NonEquidWhenTyLarge}.  We  prove Proposition~\ref{prop:MovingToFixed} and Corollary~\ref{coro:MovingToFixed}, leaving the proof of Theorem~\ref{thm:EscNeighDepOnConst} to Section~\ref{sec:EscNeighDepOnConst} .  Proposition~\ref{prop:MovingToFixed} is the most important tool in this paper.  It comes from an application of the method of stationary phase to the relevant integrals in Lemma~\ref{lemm:DoubleCosetIntegral} (see Corollary~\ref{coro:MethStatPhaseForUs}).

\begin{prop}\label{prop:MovingToFixed}  Let $j \in \{1, \cdots, q\}$, $m \in \ZZ$, $|\eta| \leq \min\left(\frac{B_1-B_0}{2}, \frac 1 4 \right)$, and $T -\min\left(\frac{B_1-B_0}{2}, \frac 1 4 \right)  >y$.  Then, as $T \rightarrow \infty$, we have that \begin{align*}\int_0^1 & \phi^{(\kappa_j)}_{T, \eta} (x+iy, 0) e(mx)~\wrt x \\\nonumber & =  O\left(\frac {1}{T^{3/2}}\right)  \\\nonumber &+ y  \sum_{\substack{\Gamma_\infty\begin{pmatrix}  a & b \\ c & d \end{pmatrix} \Gamma_\infty \in \Gamma_\infty \backslash \sigma_j^{-1}\Gamma/\Gamma_\infty \\  \quad c>0}}  \int_{-\sqrt{B} }^{\sqrt{B}}  \left(\int_0^1h(t, -2\arg(x+i))~\wrt t\right) e\left(m xy- m \frac d c\right)~\wrt x 
  \end{align*} where $B := \max\{1/Tyc^2 -1, 0\}$ and the implied constant depends on $h$, $\kappa_1$, and $\kappa_j$.  
 
\end{prop}

\begin{rema}\ \begin{enumerate}

\item In the proposition, $\eta$ is allowed to be $0$.

\item Note that, for fixed $T$ and $y$, the sum over the double cosets is finite because $B=0$ when $c$ is large.  

\item Also, note that there is a minimum strictly positive $c$ (which depends on the cusps $\kappa_1$ and $\kappa_j$) in this setup.  If $\sqrt{B}$ is strictly less than this minimum, then \[\int_0^1  \phi^{(\kappa_j)}_{T,0}(x+iy, 0) e(mx)~\wrt x=0.\]
\end{enumerate}
\end{rema}
      
As a corollary of the proposition, we have that the constant term of the Fourier expansion of $h$ with respect to $x$ gives the dominant behavior for our setup.

\begin{coro}\label{coro:MovingToFixed}
Let $j \in \{1, \cdots, q\}$, $m \in \ZZ$, $|\eta| \leq \min\left(\frac{B_1-B_0}{2}, \frac 1 4 \right)$, and $T -\min\left(\frac{B_1-B_0}{2}, \frac 1 4 \right)  >y$.  Then, as $T \rightarrow \infty$, we have that  \begin{align*} \int_0^1 & \phi^{(\kappa_j)}_{T, \eta}(x+iy, 0) e(mx)~\wrt x  =  \int_0^1  \varphi^{(\kappa_j)}_{T, \widetilde{\eta}}(x+iy, 0) e(mx)~\wrt x + O\left(\frac {1}{T^{3/2}}\right) \end{align*} where the implied constant depends on $h$, $\kappa_1$, and $\kappa_j$.  Here, $\widetilde{\eta}$ is a real number such that $|\widetilde{\eta}| \leq \min\left(\frac{B_1-B_0}{2}, \frac 1 4 \right)$.  
\end{coro}
\begin{proof}
 Apply the proposition to $\phi^{(\kappa_j)}_{T, \eta}$ and Lemma~\ref{lemm:DoubleCosetIntegral} to $\varphi^{(\kappa_j)}_{T, 0}$.  For $\widetilde{\eta} \neq 0$, the error between the application of Lemma~\ref{lemm:DoubleCosetIntegral} to $\varphi^{(\kappa_j)}_{T, 0}$ and to $\varphi^{(\kappa_j)}_{T, \widetilde{\eta}}$ is $O\left(\frac {\sqrt{|\widetilde{\eta}|}}{T^{3/2}}\right)$ where the implied constant depends on $h$, $\kappa_1$, and $\kappa_j$.  The proof is the same as the proof of the analogous additional error in the proof of Proposition~\ref{prop:MovingToFixed}.  This yields the desired result.  \end{proof}

\noindent More generally, we have 

\begin{theo}\label{thm:EscNeighDepOnConst}
 Let $j \in \{1, \cdots, q\}$, $|\eta| \leq \min\left(\frac{B_1-B_0}{2}, \frac 1 4 \right)$, $T -\min\left(\frac{B_1-B_0}{2}, \frac 1 4 \right) >y$, and $\frac 1 2 \geq \delta>0$.  Then, as $T \rightarrow \infty$, we have that  \begin{align*}\int_\alpha^\beta & \phi^{(\kappa_j)}_{T, \eta}(x+iy, 0)~\wrt x  = \int_\alpha^\beta  \varphi^{(\kappa_j)}_{T, \widetilde{\eta}}(x+iy, 0)~\wrt x + O\left(\frac {1}{T^{\delta}}\right) \end{align*} so long as $\beta - \alpha$ remains bounded and bigger than $T^{-3/4 + \delta}$.  Here $\widetilde{\eta}$ is a real number such that $|\widetilde{\eta}| \leq \min\left(\frac{B_1-B_0}{2}, \frac 1 4 \right)$, and the implied constant depends on $h$, $\kappa_1$, $\kappa_j$, and the maximum value of $\beta-\alpha$.   
\end{theo}

Finally, we show that $Ty \rightarrow 0$ is necessary in our main results, Theorems~\ref{thm:TanBunRelEqui} and~\ref{thm:EffectTanBunRelEqui}, by giving an example of a case for which $Ty \not \rightarrow 0$.

\begin{example}\label{eg:NonEquidWhenTyLarge}  For this example and Remark~\ref{rmk:NonEquidWhenTyLarge}, set $\Gamma = \PSL(2, \ZZ)$.   Define \[g_{T,0}(z, \theta) := \ind_{[T, \infty)}(y) \quad \quad \varphi_{T,0}(z, \theta) := \sum_{\gamma \in \Gamma_\infty \backslash \Gamma} g_{T,0}(\gamma(z, \theta)).\]  Let $Ty = 1/4$.  Then  \begin{align}\label{eqn:TyNotZero}
 T  \int_0^1 \varphi_{T,0}(x+iy, 0)~\wrt x = \frac{\sqrt{3}} 2 \end{align} as $T \rightarrow \infty$, but \[\frac{1}{\mu(\Gamma \backslash G)} \int_0^{2\pi}\int_0^1~\wrt x~\wrt\theta= \frac 3 \pi.\]  Note that the derivation of (\ref{eqn:TyNotZero}) follows easily from Proposition~\ref{prop:MovingToFixed}.  The gist of this derivation is to consider the left-hand side of (\ref{eqn:IdentityAreaModSurSumEuler}) {\em without} the limit but, instead, with $x$ set to $2$.

\end{example}

\begin{rema}\label{rmk:NonEquidWhenTyLarge}

If, in the example, we let $Ty \rightarrow 0$, then Theorem~\ref{thm:TanBunRelEqui} applies, from which the identity \begin{align}\label{eqn:IdentityAreaModSurSumEuler}
 \lim_{x \rightarrow \infty} 2\sum_{c=1}^{\lfloor x \rfloor}\frac{\varphi(c)} c \frac{\sqrt{x^2 - c^2}}{x^2} = \frac 3 \pi \end{align} follows by Proposition~\ref{prop:MovingToFixed}.  Here $\varphi(\cdot)$ is the Euler totient function, and $\lfloor \cdot \rfloor$ is the floor function.  Thus, we have obtained a number-theoretic identity.  The same identity will also follow from applying a version of Proposition~\ref{prop:MovingToFixed} to~\cite[Theorem~2]{St} (or the other versions of horocycle equidistribution mentioned above).  Finally, Shucheng Yu has pointed out, in personal communication, an alternative proof of (\ref{eqn:IdentityAreaModSurSumEuler}) using summation by parts and an asymptotic estimate, due to Walfisz~\cite{Wal63}, for the sum of the Euler totient function over the natural numbers from $1$ to $n$.  This same observation has been pointed out by both referees.  Furthermore, one of the referees has pointed out another alternative proof using double cosets and, in particular, using the asymptotic estimate for Kloosterman sums~\cite[Theorem~4]{Goo83} to count double cosets, namely \begin{align}\label{eqn:StrongAsyEstKloo}
 \sum_{1\leq c\leq x} \#(\textrm{double cosets indexed by } c) \sim \frac{2x^2}{\mu(\Gamma \backslash G)}  \end{align}as $x \rightarrow \infty$.  (See also~\cite[Corollary on pages 119-20]{Goo83}.)  Now applying summation by parts to (\ref{eqn:StrongAsyEstKloo}) yields (\ref{eqn:IdentityAreaModSurSumEuler}).  The author is grateful to Yu and the referees for their observations.

\end{rema}

\subsection{Proof of Proposition~\ref{prop:MovingToFixed}}

\begin{proof}[Proof of Proposition~\ref{prop:MovingToFixed}]  We first give the proof for $\phi^{(\kappa_j)}_{T,0}$.
Let $f = f_{T, 0}$ and $c>0$.  Let us compute \[\int_{-\infty}^\infty f\left(\frac a c - \frac 1{c^2 y} \frac 1 {x+i}, - 2\arg(x+i) \right) e\left(m xy- m \frac d c\right)~\wrt x.\]  Now \begin{align*} f\left(\frac a c - \frac 1{c^2 y} \frac 1 {x+i}, \theta \right) = \ind_{[T, \infty)}\left(\frac{1}{yc^2(x^2+1)}\right) h\left(\frac a c - \frac 1 {y c^2} \frac x {x^2+1}, \theta\right),
  \end{align*} which is zero if  \begin{align}\label{eqn:NonZeroCond}
 x^2 \leq 1/Tyc^2 -1 =: A\end{align} does {\em not} hold.  If $A<0$, then (\ref{eqn:NonZeroCond}) cannot hold for any value of $x$, and, thus, the integral is zero.  
 
 Otherwise, when $A \geq 0$, (\ref{eqn:NonZeroCond}) holds for some values of $x$, and, thus, the integral may be nonzero.  We now estimate the integral when $A\geq0$.  Note that, since $A \geq0$, we have \begin{align} \label{eqn:RelRatesTy} Tyc^2 \leq 1
  \end{align} holds, and we can replace the integration bounds with $-\sqrt A$ to $\sqrt A$.  By smoothness and periodicity, we have the Fourier series representation of $h$, \[h(x,\theta)  = \sum_{j \in \ZZ} \widehat{h}(j, \theta) e(-j x) = \widehat{h}(0, \theta) +\sum_{j \in \ZZ \backslash\{0\}} \frac{\widehat{h_2}(j, \theta)}{-4 \pi^2 j^2} e(-j x).\]  We consider the integral term-by-term (because the Fourier series converges uniformly to $h$).  Thus, for each $j \in \ZZ$, we need to evaluate the following integral:  \[I(T, j):=e\left(- m \frac d c - j \frac a c \right)\int_{-\sqrt{A}}^{\sqrt{A}} e^{i T p(x)}q(x)~\wrt x\] where \begin{align*}
p(x):= 2\pi j(A+1) \frac{x}{x^2+1}, \quad \quad q(x) := \widehat{h}(j, - 2\arg(x+i)) e(mxy).\end{align*}  We wish to compute the asymptotics of the integral as $T \rightarrow \infty$, uniformly for all $A \geq0$ and all $ y \geq 0$, using the method of stationary phase (see, for example,~\cite{Olv} for an introduction).  The stationary points are $x=\pm1$.   

For convenience, let us introduce the following notation.  Let $X$ be a topological space.  The function $g:X \rightarrow \CC$ is {\em locally zero at $x_0 \in X$} if there exists an open neighborhood $U$ of $x_0$ in $X$ such that $g(x) = 0$ for every point $x \in U$.

\begin{lemm}\label{lemm:MethStatPhaseForUs}  Let $j \in \ZZ \backslash \{0\}$.  We have that \[I(T, j) = I_1(T, j) + I_{-1}(T,j)\] where \begin{align*} & I_1(T) := I_1(T,j) := \\ &\begin{cases} O\left(\frac{\widehat{h}(j, -\pi/2)} {\sqrt {|j|T}} \right)  & \textrm{ if } \widehat{h}(j, -\pi/2) \neq 0 \\
O\left(\frac{\widehat{h}(j, \theta_0)} {\sqrt {|j|T}}\right) & \textrm{ if } \widehat{h}(j, -\pi/2) = 0 \textrm { and } \\&  \quad \quad  \widehat{h}(j, \theta) \textrm { is not locally zero at } \theta=-\frac{\pi} 2 \\
O\left(\frac{|\widehat{h}(j, -2\arg(\beta+i))| + |\widehat{h}(j, -2\arg(\sqrt{A}+i))|}{jT} \right) + o\left(\frac{1}{j^2 T}\right) & \textrm{ if }  \widehat{h}(j, \theta) \textrm { is locally zero at } \theta=-\frac{\pi} 2\end{cases} \\
 & I_{-1}(T) := I_{-1}(T, j) :=  \\ &\begin{cases} O\left(\frac{\widehat{h}(j, \pi/2)} {\sqrt {|j|T}} \right)  & \textrm{ if } \widehat{h}(j, \pi/2) \neq 0 \\
O\left(\frac{\widehat{h}(j, \theta_1)} {\sqrt {|j|T}}\right) & \textrm{ if } \widehat{h}(j, \pi/2) = 0 \textrm { and } \\&  \quad \quad  \widehat{h}(j, \theta) \textrm { is not locally zero at } \theta=\frac{\pi} 2 \\
O\left(\frac{|\widehat{h}(j, -2\arg(\beta+i))| + |\widehat{h}(j, -2\arg(-\sqrt{A}+i))|}{jT} \right) + o\left(\frac{1}{j^2 T}\right) & \textrm{ if }  \widehat{h}(j, \theta) \textrm { is locally zero at } \theta=\frac{\pi} 2\end{cases}
  \end{align*} as $T \rightarrow \infty$.  Here $\theta_0$ and $\theta_1$ are any values in $[0, 2 \pi)$ for which $\widehat{h}(j, \theta_0) \neq 0$ and $-1 < \beta <1$ is any value bounded away from both $-1$ and $1$.  The first implied constant of the third case for $I_{-1}(T)$ depends on $h$ and the first implied constant of the third case for $I_{1}(T)$ depends on $h$.  All other implied constants have no dependence.
 
\end{lemm}

\begin{rema} \quad \begin{itemize}
\item One can choose $\theta_0$ arbitrarily close to $-\pi/2$ and $\theta_1$ arbitrarily close to $\pi/2$ if desired.  

\item The error terms $o\left(\frac{1}{j^2 T}\right)$ can be replaced by $o\left(\frac{1}{j^\ell T}\right)$ where $\ell$ is any natural number because $h$ is a $C^\infty$-function.
\item The dependence on $h$ in the two implied constants can be made explicit by inspection of the proof below.

\item In the case that $\widehat{h}(j, \theta) \textrm { is locally zero at } \theta=\frac{\pi} 2$ and $\theta=-\frac{\pi} 2$, we can have \[I(T, j) = O\left(\frac{|\widehat{h}(j, -2\arg(\sqrt{A}+i))| + |\widehat{h}(j, -2\arg(-\sqrt{A}+i))|}{jT} \right) + o\left(\frac{1}{j^\ell T}\right)\] where where $\ell$ is any natural number.   Here, the first implied constant depends on $h$ (and can be made explicit) and the second has no dependence.

\end{itemize}
\end{rema}
\begin{proof}  We assume that $j \geq 1$.  The proof for $j \leq -1$ is analogous.  We break $I(T, j)$ into four (or, if $A\leq1$, two) integrals so that $\{\pm 1, \pm A\}$ appears as only one endpoint (and the other endpoint of each integral does not matter).  Let us for now assume that $\widehat{h}(j, \pi/2) \neq 0$ and $\widehat{h}(j, -\pi/2) \neq 0$. We will remove this assumption at the end.  There are a number of cases of which we begin with the case $1 < \sqrt{A}$.  The stationary points (except in a special case of $\widehat{h}(j, -2\arg(x+i))$ being locally zero at both $-1$ and $1$, detailed below) will contribute the main term, and we consider them first.  Let \[I_1^+(T):= e\left(- m \frac d c - j \frac a c \right)\int_1^\alpha e^{i T p(x)}q(x)~\wrt x\] where $\alpha := \sqrt{A}$.   Using Taylor series for $-p(x)$ and $q(x)$, we have that, as $x \rightarrow 1^+$, \[-p(x)+p(1) \sim \frac 1 2 \pi j (A+1)(x-1)^2, \quad \quad q(x) \sim \widehat{h}(j, -\pi/2)e(my)\] and, thus, the function $q(x)/p'(x)$ is of bounded variation over every closed interval $[k, \alpha]$ where $k \in (1, \alpha)$ because both the numerator and denominator are of bounded variation and the denominator is bounded away from zero.  Now we may apply~\cite[Chapter~3, Theorem~13.1]{Olv} to obtain \[I_1^+(T) \sim \frac{\sqrt{\pi}} 2 e\left(- m \frac d c - j \frac a c \right)  e^{-\pi i /4} \widehat{h}(j, -\pi/2) e(my)\frac{e^{iTp(1)}}{(\frac 1 2 \pi j (A+1)T)^{1/2}}  \] as $T \rightarrow \infty$ uniformly for all $A >1$ and $y \geq 0$.

Now substituting $A = 1/Tyc^2 -1$ from (\ref{eqn:NonZeroCond}) yields \begin{align} \label{eqn:statPointContrib}
 I_1^+(T) \sim \frac{1} {\sqrt{2}} e\left(- m \frac d c - j \frac a c + j \frac 1 {2 y c^2} - \frac 1 8\right) \widehat{h}(j, -\pi/2)e(my) \frac{c \sqrt y}{\sqrt{ j}}\end{align} as $T \rightarrow \infty$.  Applying (\ref{eqn:RelRatesTy}) gives \[I_1^+(T)= O\left(\frac{\widehat{h}(j, -\pi/2)} {\sqrt {jT}} \right)\] as $T \rightarrow \infty$.  Here the implied constant has no dependence.

Let \[I_1^-(T):= e\left(- m \frac d c - j \frac a c \right)\int_\alpha^1 e^{i T p(x)}q(x)~\wrt x\] where $-1 < \alpha <1$. Then the analogous proof gives \[I_1^-(T)= O\left(\frac{\widehat{h}(j, -\pi/2)} {\sqrt {jT}} \right)\] as $T \rightarrow \infty$.  In fact, (\ref{eqn:statPointContrib}) holds when $I^+_1(T)$ is replaced by $I^-_1(T)$.  Let $I^{+}_{-1}(T)$ and $I^-_{-1}(T)$ denote the analogous integrals for the stationary point $-1$, then the analogous proof gives \[I_{-1}^+(T)= O\left(\frac{\widehat{h}(j, \pi/2)} {\sqrt {jT}} \right) = I_{-1}^-(T)\] as $T \rightarrow \infty$.  The implied constants for these latter three integral also have no dependence.

When $A =1$, only the integrals $I_1^-(T)$ and  $I^{+}_{-1}(T)$ are needed.

For $0\leq A < 1$, we must consider the two analogous integrals $I_{\sqrt{A}}^-(T)$ and  $I^{+}_{-\sqrt{A}}(T)$.  To obtain a result uniform in $A \geq 0$, we write \[I_{\sqrt{A}}^-(T) =  e\left(- m \frac d c - j \frac a c \right)\left(\int_\alpha^1 e^{i T p(x)}q(x)~\wrt x - \int_{\sqrt{A}}^1 e^{i T p(x)}q(x)~\wrt x\right),\] from which it immediately follows that  \[I_{\sqrt{A}}^-(T)= O\left(\frac{\widehat{h}(j, -\pi/2)} {\sqrt {jT}} \right)\] as $T \rightarrow \infty$.  Similarly, for $I^{+}_{-\sqrt{A}}(T)$, we have that  \[I_{-\sqrt{A}}^+(T)= O\left(\frac{\widehat{h}(j, \pi/2)} {\sqrt {jT}} \right)\] as $T \rightarrow \infty$.  The implied constants for these two integrals have no dependence.  This proves the desired result in the cases for which both $\widehat{h}(j, \pi/2) \neq 0$ and $\widehat{h}(j, -\pi/2) \neq 0$.

Finally, consider when $\widehat{h}(j, -\pi/2) = 0$ or $\widehat{h}(j, \pi/2) = 0$.  As we have seen, the integrals $I_1^{\pm}(T)$ and $I_{\sqrt{A}}^{-}(T)$ depend on $\widehat{h}(j, -\pi/2)$ and the integrals $I_{-1}^{\pm}(T)$ and $I_{-\sqrt{A}}^{+}(T)$ depend on $\widehat{h}(j, \pi/2)$.  Let us assume that $\widehat{h}(j, -\pi/2) = 0$. There are two cases to consider.

The first case is that $\widehat{h}(j, - 2\arg(x+i))$ is not locally zero at $x=1$.  Thus, we have that, for every $\delta>0$, there exists $x \in (1 -\delta, 1+\delta)$ such that $\widehat{h}(j, - 2\arg(x+i)) \neq  0$.  Let $\theta_0$ be chosen such that  $\widehat{h}(j, \theta_0) \neq  0$.  We note that \[I_1^{+}(T)= e\left(- m \frac d c - j \frac a c \right) \left( \int_1^{\sqrt{A}} e^{i T p(x)}\left(q(x)+\widehat{h}(j, \theta_0) \right)~\wrt x - \int_1^{\sqrt{A}} e^{i T p(x)}\widehat{h}(j, \theta_0) ~\wrt x\right).\]  Apply the above proof for $I_1^+(T)$, we obtain \[I_1^+(T)= O\left(\frac{\widehat{h}(j, \theta_0)} {\sqrt {jT}} \right)\] as $T \rightarrow \infty$.  Similarly, we have \[I_1^-(T)= O\left(\frac{\widehat{h}(j, \theta_0)} {\sqrt {jT}} \right) = I_{\sqrt{A}}^{-}(T)\] as $T \rightarrow \infty$.  The implied constants have no dependence for any of these estimates.  Note that~\cite[Chapter~3, Theorem~13.2]{Olv} does not apply to the first case as the hypothesis on bounded variation does not hold.

The second case is that $\widehat{h}(j, - 2\arg(x+i))$ is locally zero at $x=1$.  Thus, we have that, there exists a $\delta>0$ for which $\widehat{h}(j, - 2\arg(x+i)) = 0$ for every $x \in (1 -\delta, 1+\delta)$.  Let $N$ denote the union of all open intervals $U$ containing $1$ such that $\widehat{h}(j, - 2\arg(x+i)) = 0$ for every $x \in U$.  Since we are in the second case, $N$ is not the empty set.  Thus, we have $N =:(r, s)$ where $r<s$ are unique values determined solely by $h$.  Here $r$ may be $-\infty$ and $s$ may be $\infty$.  We note that $\widehat{h}(j, - 2\arg(x+i)) = 0$ for every $x$ in the topological closure $\overline{N}$.  If $s$ is finite, then there exists an interval $(s, s_+)$ such that $\widehat{h}(j, - 2\arg(x+i)) \neq 0$ for every $x \in(s, s_+)$.  Likewise, if $r$ is finite, then there exists an interval $(r_-, r)$ such that $\widehat{h}(j, - 2\arg(x+i)) \neq 0$ for every $x \in(r_-, r)$.

Let us consider the case $1 < \sqrt{A}$ first. If $N \supset (1, \sqrt{A})$, then $I_1^+(T)=0$.  Otherwise, we have that $\sqrt{A}>s >1$.  Let $x_0 \in(s, s_+)$.   We note that \[I_1^+(T)=  e\left(- m \frac d c - j \frac a c \right)\left(\int_s^{\sqrt{A}} e^{i T p(x)} q_1(x)~\wrt x  +\int_s^{\sqrt{A}} e^{i T p(x)}q_2(x)~\wrt x \right)\] where $q_1(x):= q(x) - q(x_0)$ and $q_2(x) = q(x_0)$.  Using Taylor series, we have, as $x \rightarrow s^+$, \[-p(x)+p(s) \sim 2\pi j (A+1)\frac{s^2-1}{(s^2+1)^2}(x-s), \quad \quad q_1(x) \sim -q(x_0), \quad \quad q_2(x) \sim q(x_0).\]  We may now apply~\cite[Chapter~3, Theorem~13.2]{Olv} to obtain, as $T \rightarrow \infty$, \[I_1^+(T)= e\left(- m \frac d c - j \frac a c \right)\left(\frac{q(\sqrt{A})e^{iTp(\sqrt{A})}}{i T p'(\sqrt{A})} + \varepsilon_1(T) + \varepsilon_2(T)\right)\] where \begin{align*}
\varepsilon_1(T) :=& \frac{1}{i T}\int_s^{\sqrt{A}}e^{iTp(x)} \frac{\wrt{}}{\wrt x}\left(\frac{q_1(x)}{-p'(x)} \right)~\wrt x \\ =& -\frac{1}{4 \pi^2 i j^2 T}\int_s^{\sqrt{A}}e^{iTp(x)} \frac{\wrt{}}{\wrt x}\left(\frac{\widehat{h_2}(j, -2 \arg(x+i))e(mxy)-\widehat{h_2}(j, -2 \arg(x_0+i))e(mx_0y)}{-p'(x)} \right)~\wrt x \\ =& o\left(\frac{1}{j^2 T}\right), \\
\varepsilon_2(T) :=& \frac{1}{i T}\int_s^{\sqrt{A}}e^{iTp(x)} \frac{\wrt{}}{\wrt x}\left(\frac{q_2(x)}{-p'(x)} \right)~\wrt x \\ =& -\frac{1}{4 \pi^2 i j^2 T}\int_s^{\sqrt{A}}e^{iTp(x)} \frac{\wrt{}}{\wrt x}\left(\frac{\widehat{h_2}(j, -2 \arg(x_0+i))e(mx_0y)}{-p'(x)} \right)~\wrt x  \\=& o\left(\frac{1}{j^2 T}\right). \end{align*}  Here the three implied constants have no dependence.  Note the expressions for $\varepsilon_1(T)$ and $\varepsilon_2(T)$ follow immediately from integration by parts for $I_1^+(T)$, and the third equality for both $\varepsilon_1(T)$ and $\varepsilon_2(T)$ follow from the Riemann-Lebesgue lemma (see~\cite[Chapter~3, (13.03) and Section~13.3]{Olv}).  Consequently, we have, as $T \rightarrow \infty$, \[I_1^+(T)= O\left(\frac{\widehat{h}(j, -2\arg(\sqrt{A}+i))}{jT} \right) + o\left(\frac{1}{j^2 T}\right).\] Note that, since $A > s^2>1$, we have that $-2 \pi j < p'(\sqrt{A})< 2 \pi j \left(\frac 2 {s^2+1} -1 \right)$ is bounded away from zero.  Here the first implied constant depends on $s$ and the second implied constant has no dependence.  As $s$ is uniquely determined by $h$, we can also state that the first implied constant depends on $h$.

Next we consider the integral $I_1^-(T)$ for $\alpha$ (which, in our notation, denotes the lower integration bound) such that $-1 < \alpha < 1$ is bounded away from $-1$.  If $N \supset (\alpha, 1)$, then $I_1^-(T)=0$.  Otherwise, we have that $\alpha < r <1$.  Now, by the analogous proof for $I_1^+(T)$, we have, as $T \rightarrow \infty$,  \[I_1^-(T)= e\left(- m \frac d c - j \frac a c \right)\left(\frac{q(\alpha)e^{iTp(\alpha)}}{-i T p'(\alpha)} + \varepsilon(T) \right)\] where $\varepsilon(T) = o\left(\frac{1}{j^2 T}\right)$. Here the implied constants have no dependence.  Now, as $-Tp'(\alpha) = 2 \pi j \frac{\alpha^2-1}{yc^2(\alpha^2+1)^2}$, we have, by (\ref{eqn:RelRatesTy}), that \[I_1^-(T)= O\left(\frac{\widehat{h}(j, -2\arg(\alpha+i))}{jT} \right) + o\left(\frac{1}{j^2 T}\right)\]  as $T \rightarrow \infty$.  Here the first implied constant depends on $r$ and the second has no dependence.  As $r$ is uniquely determined by $h$, we can also say that the first constant depends on $h$.

Now, for the case $\sqrt{A} =1$, we only need to consider the integral $I_1^-(T)$ as above for the case $\sqrt{A} > 1$.

For the case $\sqrt{A} <1$, we only need to consider the integral $I_{\sqrt{A}}^-(T)$ for $\alpha$ such that $-1 < \alpha < \sqrt{A}$ is bounded away from $-1$.  If $N \supset (\alpha, \sqrt{A})$, then $I_{\sqrt{A}}^-(T)=0$.  Otherwise, we have that $\alpha < r <\sqrt{A}$.  Now, by the analogous proof to the case $\sqrt{A} > 1$, we have, as $T \rightarrow \infty$, \[I_{\sqrt{A}}^-(T)= O\left(\frac{|\widehat{h}(j, -2\arg(\alpha+i))| + |\widehat{h}(j, -2\arg(\sqrt{A}+i))|}{jT} \right) + o\left(\frac{1}{j^2 T}\right)\] where the first implied constant depends on $r$ and the second has no dependence.  

Finally, we consider when $\widehat{h}(j, \pi/2) = 0$.  As for $\widehat{h}(j, -\pi/2) = 0$, there are two cases.  The first case is that $\widehat{h}(j, - 2\arg(x+i))$ is not locally zero at $x=-1$. Thus, we have that, for every $\delta>0$, there exists $x \in (-1 -\delta, -1+\delta)$ such that $\widehat{h}(j, -2\arg(x+i)) \neq  0$.  Let $\theta_1$ be chosen such that  $\widehat{h}(j, \theta_1) \neq  0$. By the proofs analogous to when $\widehat{h}(j, -\pi/2) = 0$, we  have \[I_{-1}^{\pm}(T)=  O\left(\frac{\widehat{h}(j, \theta_1)} {\sqrt {jT}} \right)  = I_{\sqrt{A}}^{-}(T)\] as $T \rightarrow \infty$.  The implied constants have no dependence for any of these estimates.  

The second case is that $\widehat{h}(j, - 2\arg(x+i))$ is locally zero at $x=-1$.  Thus, we have that there exists a $\delta>0$ for which $\widehat{h}(j, - 2\arg(x+i)) = 0$ for every $x \in (-1 -\delta, -1+\delta)$.  Let $N$ denote the union of all open intervals $U$ containing $-1$ such that $\widehat{h}(j, - 2\arg(x+i)) = 0$ for every $x \in U$ and set $N =:(r, s)$.  By the analogous proof for when $\widehat{h}(j, -\pi/2) = 0$, we have, as $T \rightarrow \infty$, \begin{align*}
I_{-1}^-(T) &= O\left(\frac{\widehat{h}(j, -2\arg(-\sqrt{A}+i))}{jT} \right) + o\left(\frac{1}{j^2 T}\right) \\ 
I_{-1}^+(T)&= O\left(\frac{\widehat{h}(j, -2\arg(\alpha+i))}{jT} \right) + o\left(\frac{1}{j^2 T}\right) \\
I_{-\sqrt{A}}^+(T)&= O\left(\frac{|\widehat{h}(j, -2\arg(\alpha+i))| + |\widehat{h}(j, -2\arg(-\sqrt{A}+i))|}{jT} \right) + o\left(\frac{1}{j^2 T}\right) \end{align*} for $-1 < \alpha < 1$ bounded away from $1$.  For $I_{-1}^-(T)$, the first implied constant depends on $r$ and the second has no dependence.  For $I_{-1}^+(T)$ and $I_{-\sqrt{A}}^+(T)$, the first implied constant depends on $s$ and the second has no dependence. This concludes all cases and proves the desired result. \end{proof}

\begin{coro}\label{coro:MethStatPhaseForUs}  Let $j \in \ZZ \backslash \{0\}$.  We have, as $T \rightarrow \infty$, \[I(T, j)=\begin{cases}  O\left(\frac{1}{j^2T} \right)  &\textrm{ if } \widehat{h}(j, \theta) \textrm { is locally zero at both } \theta=-\frac{\pi} 2 \textrm{ and }  \theta=\frac{\pi} 2 \\ 
O\left(\frac{1} {j^2\sqrt {T}}\right) &\textrm{ otherwise}\end{cases}\] where the implied constants depend only on $h$.  
\end{coro}

\begin{proof}
 By the smoothness of $h$, we can apply $\widehat{h}(j, \theta) = \frac{\widehat{h_2}(j, \theta)}{-4 \pi^2 j^2}$ to the lemma to obtain the desired result.
\end{proof}

Applying the corollary, we have that $\sum_{j \in \ZZ\backslash \{0\}} I(T, j) = O(1/\sqrt{T})$ where the implied constant depends on $h$.  Thus, we have, as $T \rightarrow \infty$, \begin{align}\label{eqn:LargeTFourierConstDom}\int_{-\infty}^\infty & f\left(\frac a c - \frac 1{c^2 y} \frac 1 {x+i}, - 2\arg(x+i) \right) e\left(m xy- m \frac d c\right)~\wrt x \\\nonumber &= \int_{-\sqrt{A}}^{\sqrt{A}}  \left(\int_0^1h(t, -2\arg(x+i))~\wrt t\right) e\left(m xy- m \frac d c\right)~\wrt x  + O\left(\frac {H(A)}{\sqrt{T}}\right)
  \end{align} where \[H(A):= \begin{cases} 1 & \text{ if } A>0 \\ 0 &\text{ if } A =0  \end{cases}\] and the implied constant depends on $h$.  This finishes the estimate for the integral when $A \geq 0$, and, thus, in all cases.  Note, in particular, for $A=0$, the all the integrals we considered are zero including those which give the error term.
  
  Now the number of terms in the sum over the double cosets is counted by a suitable Kloosterman sum (see~\cite[(2.24)]{Iwa02}).  Using the trivial bound over the average of the Kloosterman sum, namely~\cite[(2.38)]{Iwa02}, and H\"older's inequality, we have \begin{align*} y  \sum_{\substack{\Gamma_\infty\begin{pmatrix}  a & b \\ c & d \end{pmatrix} \Gamma_\infty \in \Gamma_\infty \backslash \sigma_j^{-1}\Gamma/\Gamma_\infty \\  \quad \frac 1 {\sqrt{Ty}} \geq c>0}}  \frac 1 {\sqrt{T}} \leq C\frac{y}{Ty \sqrt{T}} = C \frac 1{T^{3/2}}
  \end{align*} where $C$ is a constant.\footnote{Our technique of using the trivial bound over the average of the Kloosterman sum here and elsewhere could be replaced by a more precise estimate~\cite[Theorem~4]{Goo83}:\[\sum_{\substack{\Gamma_\infty\begin{pmatrix}  a & b \\ c & d \end{pmatrix} \Gamma_\infty \in \Gamma_\infty \backslash \sigma_j^{-1}\Gamma/\Gamma_\infty \\  \quad \frac 1 {\sqrt{Ty}} \geq c>0}} 1 = \frac{2}{\mu(\Gamma\backslash G)}(Ty)^{-1} +O\left((Ty)^{-\max\{s_1, 2/3\}}\right).\]  This shows that $C=\frac{2}{\mu(\Gamma\backslash G)}$.  For the definition of $s_1$, see Section~\ref{sec:HorocycleEquiFixedFunc}.}   Note that terms corresponding to $c > \frac 1 {\sqrt{Ty}}$ in the double coset summation are identically zero.  This yields the desired error term.

  Finally, as $T >y$, we have that \[\int_0^1 f(x+iy, 0) \ e(m x)~\wrt x =0.\]  Apply Lemma~\ref{lemm:DoubleCosetIntegral} yields the desired result for $\phi^{(\kappa_j)}_{T,0}$.
  
We now consider the case $\eta>0$.  We claim, for this case, that there is an additional error term $O\left(\frac{\sqrt{\eta}}{T^{3/2}}\right)$, which, thus, is negligible.  Here, the implied constant depends on $h$, $\kappa_1$, and $\kappa_j$.  Let $f_\eta := f_{T,\eta}$ and $g := f - f_\eta$.  Similar to the case $\eta =0$, the $c=0$ term is zero.  Let $A_\eta:=\frac1{(T-\eta)yc^2} -1$.  For any other term, we compute \begin{align*}\left|\int_{-\infty}^\infty g\left(\frac a c - \frac 1{c^2 y} \frac 1 {x+i}, - 2\arg(x+i) \right) e\left(m xy- m \frac d c\right)~\wrt x \right| \\ \leq 2 M \begin{cases}  \sqrt{A_\eta} - \sqrt{A} & \textrm{ if } A > \frac \eta T \\ \sqrt{A_\eta} & \textrm{ if } A \leq \frac \eta T \textrm{ and } A_\eta \geq 0\\
0 & \textrm{ if } A_\eta < 0 \end{cases}\end{align*}  where \[M := \max \{|h(x, \theta)| : x \in [0,1] \textrm { and } \theta \in [0, 2\pi]\}.\]  We note that $A_\eta > A$ here.

Let us consider first $A > \frac{\eta} {T}$.  From this, it follows that $\frac 1 {Tyc^2} \frac T {T +\eta} > 1$ and, thus, $\sqrt{A} > \frac{\sqrt{\eta}}{\sqrt{2} T \sqrt{y} c}$.  Now we have \begin{align*}
\sqrt{A_\eta} - \sqrt{A} = \frac{A_\eta - A}{\sqrt{A_\eta} +\sqrt{A}} \leq \frac{\eta}{2\sqrt{A} T(T-\eta) y c^2} \leq \frac{4\sqrt{\eta}}{T\sqrt{y} c}.\end{align*}

For $A \leq \frac{\eta} {T}$ and $A_\eta \geq 0$, we have that $A_\eta \leq \frac{2 \eta}{T-\eta}$ and $1 \leq \frac 1{\sqrt{T-\eta} \sqrt y c}$, from which it follows that $\sqrt{A_\eta} \leq \frac{4\sqrt{\eta}}{T\sqrt{y} c}$.

Now any term for which $c> \frac 1 {\sqrt{(T-\eta) y}}$ is zero.  Using the trivial bound over the average of the Kloosterman sum (\cite[(2.38)]{Iwa02}), we have that the additional error term is $ y \frac{4\sqrt{\eta}}{T\sqrt{y}} \frac 1 {\sqrt{(T-\eta) y}} = O\left(\frac{\sqrt{\eta}}{T^{3/2}}\right)$, as desired.  Finally, for the case $\eta <0$, the analogous proof gives an additional error of $O\left(\frac{\sqrt{|\eta|}}{T^{3/2}}\right)$.  This completes the proof of Proposition~\ref{prop:MovingToFixed}.  \end{proof}

\section{Renormalizing a shrinking neighborhood to a fixed neighborhood}\label{sec:RenormShrinkNeigh}  In this section, we prove Proposition~\ref{prop:MovingToFixed2}, which, roughly speaking, allows us to renormalize a neighborhood, namely the support of our test function $\phi^{(\kappa_j)}_{T, \eta}$, shrinking into the cusp $\kappa_j$ as $T \rightarrow \infty$ to a fixed neighborhood, namely the support of a suitable auxiliary function.

\begin{prop}\label{prop:MovingToFixed2}  Let $j \in \{1, \cdots, q\}$, $\frac 1 2 >\delta>0$, $|\eta| \leq \min\left(\frac{B_1-B_0}{2}, \frac 1 4 \right)$, $T -\min\left(\frac{B_1-B_0}{2}, \frac 1 4 \right)  >y$, $0<Ty<B^2_1-\frac 1 4$, and $m \in \ZZ$ be such that $|m| \leq \left(\frac{Ty}{B_1}\right)^{-1/2+\delta}$.  Then, as $T \rightarrow \infty$ and $Ty \rightarrow 0$, we have \begin{align*}
T \int_0^1 & \phi^{(\kappa_j)}_{T, \eta}(x+iy, 0) e(mx)~\wrt x \\&=  B_1\left(1 + O\left(\frac{Ty}{B_1}\right)^{\delta} \right) \int_0^1 \phi^{0, (\kappa_j)}_{\widetilde{\eta}}\left(x+i\frac{T}{B_1}y, 0\right) e(mx)~\wrt x    + O\left(\frac {1}{\sqrt{T}}\right) \end{align*} where $\widetilde{\eta}$ is a real number such that $|\widetilde{\eta}| \leq \min\left(\frac{B_1-B_0}{2}, \frac 1 4 \right)$.   Here, the first implied constant depends on $\kappa_1$ and $\kappa_j$ and the second depends on $h$, $\kappa_1$, and $\kappa_j$.   
 
\end{prop}

\begin{proof}  We will show the desired result by equating each term of the summation over the double coset decomposition given by Lemma~\ref{lemm:DoubleCosetIntegral}.  It is known that the decomposition is uniquely determined by $c$ and $d(\bmod c)$~\cite[Theorem~2.7]{Iwa02}.  We first prove the proposition for $\phi_T:= \phi_{T, 0}^{(\kappa_j)}$ and $\phi^0:= \phi^{0,(\kappa_j)}_0$.  Let $f_T := f_{T, 0}$ and $f^0:=f^0_0$.  Proposition~\ref{prop:MovingToFixed} gives the desired expression for \[\int_0^1 \phi_T(x+iy, 0) e(mx)~\wrt x.\]   

We now compute the desired expression for \[\int_0^1 \phi^0\left(x+i\frac{T}{B_1}y, 0\right) e(mx)~\wrt x.\]  As $Ty<B^2_1-\frac 1 4$, we have that the $c=0$ term is equal to zero, as desired.   Now let $c >0$ and consider the term in the summation determined uniquely by $c$ and $d(\mod c)$.  Let $B := \max\{1/Tyc^2 -1, 0\}$.  We have that \begin{align*}
\frac{T}{B_1}y\int_{-\infty}^\infty  &f^0\left(\frac a c - \frac {B_1}{c^2 Ty} \frac 1 {x+i}, - 2\arg(x+i) \right) e\left(m x\frac{T}{B_1}y- m \frac d c\right)~\wrt x \\  &=\frac{T}{B_1}y \int_{-\sqrt{B} }^{\sqrt{B}}  \left(\int_0^1h(t, -2\arg(x+i))~\wrt t\right) e\left(m x \frac{T}{B_1}y- m \frac d c\right)~\wrt x 
 \\ & = \frac{T}{B_1}y \left(1 + O\left(\frac{Ty}{B_1}\right)^{\delta} \right)e\left(- m \frac d c\right)  \int_{-\sqrt{B} }^{\sqrt{B}}  \left(\int_0^1h(t, -2\arg(x+i))~\wrt t\right) ~\wrt x \end{align*} as $Ty \rightarrow 0$.  Note that, over the integration bounds, we have that $|2 \pi i m x \frac{T}{B_1} y| \leq 2 \pi \left(\frac{Ty}{B_1}\right)^{\delta}$.  Here, the implied constant has no dependence.
 
By Proposition~\ref{prop:MovingToFixed}, the corresponding term in the double coset decomposition for $\phi_T$ is \begin{align*}
y\int_{-\infty}^\infty  f_T& \left(\frac a c - \frac 1{c^2 y} \frac 1 {x+i}, - 2\arg(x+i) \right) e\left(m xy- m \frac d c\right)~\wrt x \\ &=y\int_{-\sqrt{B} }^{\sqrt{B}}  \left(\int_0^1h(t, -2\arg(x+i))~\wrt t\right) e\left(m xy- m \frac d c\right)~\wrt x  
\\ &=y\left(1 + O\left(\frac{Ty}{B_1}\right)^\delta \right)e\left(- m \frac d c\right)\int_{-\sqrt{B} }^{\sqrt{B}}  \left(\int_0^1h(t, -2\arg(x+i))~\wrt t\right)~\wrt x\end{align*}  as $T \rightarrow \infty$.  Here, the implied constant depends on $\kappa_1$ and $\kappa_j$.  For $B >0$, (\ref{eqn:NonZeroCond}) in the proof of Proposition~\ref{prop:MovingToFixed} applies.  Consequently, over the integration bounds $-\sqrt{B}$ to $\sqrt{B}$, we have $|2 \pi i m x y| \leq O\left(\frac{\sqrt{B_1}}{T}\left(\frac{Ty}{B_1}\right)^\delta\right) \leq O\left(\frac{Ty}{B_1}\right)^\delta$, where the implied constants depend on $\kappa_1$ and $\kappa_j$.   For $B=0$, this term, the corresponding term in the other summation, and the error term from Proposition~\ref{prop:MovingToFixed} for $\phi_{T, 0}^{(\kappa_j)}$ are all equal to zero.

Thus, as $T \rightarrow \infty$ and $Ty \rightarrow 0$, we have that \begin{align*}
\frac{T}{B_1}y\int_{-\infty}^\infty & f_T \left(\frac a c - \frac 1{c^2 y} \frac 1 {x+i}, - 2\arg(x+i) \right) e\left(m xy- m \frac d c\right)~\wrt x  \\ & = \frac{T}{B_1}y\left(1 + O\left(\frac{Ty}{B_1}\right)^{\delta} \right) \\ & \quad \quad \int_{-\infty}^\infty  f^0 \left(\frac a c - \frac {B_1}{c^2 Ty} \frac 1 {x+i}, - 2\arg(x+i) \right) e\left(m x\frac{T}{B_1}y- m \frac d c\right)~\wrt x \end{align*} where the implied constant depends on $\kappa_1$ and $\kappa_j$.  Adding in the error term of $O\left(\frac 1 {B_1\sqrt{T}}\right)$, we have, thus, shown the desired equality of each corresponding term in the double coset decompositions.  Here, the implied constant depends on $h$, $\kappa_1$, and $\kappa_j$.  This proves the desired result when $\eta =0 =\widetilde{\eta}$.

For $\eta \neq 0$, we must add another error term of $O\left(\frac{\sqrt{|\eta|}}{B_1 \sqrt{T}}\right)$ as shown in the proof of Proposition~\ref{prop:MovingToFixed}.  For $\widetilde{\eta} \neq 0$, we must add in yet another error term, namely $O\left(\frac{\sqrt{|\widetilde{\eta}|}}{B_1 \sqrt{T}}\right)$.  This follows because the functions \[\ind_{[T- \widetilde{\eta}, \infty)}\circ\mathfrak{m}_{\frac T {B_1}}\left(\frac{B_1}{Tyc^2(x^2+1)} \right) \quad \quad \ind_{[T- \widetilde{\eta}, \infty)}\left(\frac{1}{yc^2(x^2+1)} \right)\] are the same.  Then applying the proof of the additional error term in Proposition~\ref{prop:MovingToFixed} (with the $\eta$ from the proposition replaced by $\widetilde{\eta}$) yields the desired error term.  Both of the implied constants depend on $h$, $\kappa_1$, and $\kappa_j$.  This proves the desired result in all cases.  \end{proof}

\section{Horocycle equidistribution for fixed functions}\label{sec:HorocycleEquiFixedFunc}  In this section, we give a variant of a general formulation of horocycle equidistribution for fixed functions due to Str\"ombergsson.  There are two formulations, an effective one and a non-effective one.   Str\"ombergsson's results are for continuous functions, while, in our setting, the non-effective version also needs to handle functions that are not continuous, as these functions involve indicator functions.  A simple approximation is needed to prove our  variant of the non-effective result (Lemma~\ref{lemm:HoroDistTanBun}), namely use the smooth Urysohn lemma (see~\cite[Lemma~2.1.17]{Muk15} for example) to give an upper and lower approximation for the indicator function and apply Str\"ombergsson's results.  For the convenience of the reader, we will give the details in the proof of Lemma~\ref{lemm:HoroDistTanBun} below.

First, let us prove our variant of the effective result, Lemma~\ref{lemm:EffectHoroDistTanBun}.  Let $\Delta$ be the Laplacian on $\Gamma \backslash \HH$.  If there exist small eigenvalues $\lambda \in (0, \frac 1 4)$ in the discrete spectrum of $\Delta$, let $\lambda_1$ be the smallest.  Likewise, if it exists, let $\lambda^{(j)}_1 \in [\lambda_1, \frac 1 4)$ be the smallest positive eigenvalue for which there exists an eigenfunction which has non-zero constant term (see~\cite[Chapter~7]{Bor97}) at cusp $\kappa_j$.  Define \begin{align}\label{eqn:SmallEigenvalues}
s_1 &:= \begin{cases} \frac{1 + \sqrt{1-4 \lambda_1}}2 &\textrm{ if there exists small eigenvalues} \\
\frac 1 2 &\textrm{ if there does not exist small eigenvalues}  \end{cases} \\\nonumber
s_1^{(j)} &:= \begin{cases} \frac{1 + \sqrt{1-4 \lambda^{(j)}_1}}2 &\textrm{ if } \lambda^{(j)}_1 \textrm{ exists } \\
\frac 1 2 &\textrm{ if } \lambda^{(j)}_1 \textrm{ does not exist } 
 \end{cases}\\\nonumber
 s_1' &:= \max\left(s_1^{(1)}, \cdots, s_1^{(q)}\right).
 \end{align}  Note that $\frac 1 2 \leq s'_1 \leq s_1<1$.
 
 Our variant of the results in~\cite{Str13} is the following:

\begin{lemm}\label{lemm:EffectHoroDistTanBun} Let $j \in \{1, \cdots, q\}$, $\delta>0$, and $|\eta|>0$.  Then \begin{align*}
\frac 1 {\beta - \alpha} \int_\alpha^\beta & \phi^{0, (\kappa_j)}_{\eta}(x+iy, 0)~\wrt x 
 = \frac{T}{B_1\mu(\Gamma \backslash G)} \int_0^{2\pi}\int_0^1\int_0^\infty \ind_{[T,\infty), \eta}(y) h(x, \theta) \frac{\wrt y~\wrt x~\wrt\theta}{y^2}
   \\ &+ O(B_1^{-1}T^4|\eta|^{-4}) \left(\frac{y^{1/2}}{\beta - \alpha} \log^2 \left(\frac{\beta-\alpha}{y}\right) + \frac{y^{1-s'_1}}{(\beta-\alpha)^{2(1-s'_1)}}+\left(\frac{y}{\beta-\alpha} \right)^{1-s_1}\right)\end{align*} where the implied constant depends on $\Gamma, \kappa_1, \kappa_j$ and $h$. 
 
\end{lemm}

\begin{rema}  The main term is bounded for all $T>\frac 1 4$.   For fixed $T$, the error term goes uniformly to $0$ as $y \rightarrow 0$ so long as $1 \geq \beta - \alpha \geq y^{1/2 - \delta}$.  Moreover, the implied constant in the error term is effective, meaning that, in principle, it can be computed from the proof of~\cite[Proposition~3.1]{Str13} and the proof of the lemma.

\end{rema}
\begin{proof}
 The proof is applying~\cite[Proposition~3.1]{Str13} (c.f.~\cite[Remark~3.4]{Str13}) to $f:=\phi^{0, (\kappa_j)}_{\eta}$.  The details are as follows.  Let \[\boldsymbol{n}(t) := \begin{pmatrix} 1 & t \\ 0 & 1  \end{pmatrix} \quad \boldsymbol{a}(y):=\begin{pmatrix} y^{1/2} & 0 \\ 0 & y^{-1/2} \end{pmatrix}.\]  As $\boldsymbol{n}(\alpha) \boldsymbol{a}(y)\boldsymbol{n}(t)i= yt+\alpha +i y$ by(\ref{eqn:MobiusAct}), we have \[\frac 1 T \int_0^Tf(\boldsymbol{n}(\alpha) \boldsymbol{a}(y)\boldsymbol{n}(t)) = \frac 1 {\beta-\alpha} \int_\alpha^\beta f(x+iy,0)~\wrt x\] and $\beta-\alpha = yT$.  
 
To finish, it remains only to estimate the Sobolev norm $\|f\|_{W_4}$.  We first show that it is finite.  Let $\pi$ denote the right regular representation of $G$ on $L^2(\Gamma \backslash G)$, $\mathfrak{g}$ denote the Lie algebra of $\SL_2(\RR)$, and $\mathcal{U}(\mathfrak{g})$ denote the universal enveloping algebra of $\mathfrak{g}$.  Then the action of $\mathfrak{g}$ on the smooth functions of $L^2(\Gamma \backslash G)$ is by the (left) Lie derivative, and, for a fixed basis $X_1, X_2, X_3$ of $\mathfrak{g}$ and all $k \in \NN$, the Sobolev norm $\|f\|_{W_k}$ is equivalent to the norm given by $\sum\|\pi(X_\zeta)f\|$ where the sum runs over all monomials $X_\zeta = X_{i_1} X_{i_2} \cdots X_{i_l} \in \mathcal{U}(\mathfrak{g})$ of degree $\leq k$ and the norm $\|\cdot\|$ is the $L^2$-norm (see~\cite[Section~2]{Str13},~\cite[Section~3.2]{BR99}, and~\cite[Chapter~14]{Bor97}).

A standard basis for $\mathfrak{g}$ is \[H:=\begin{pmatrix} 1 & 0 \\ 0 & -1  \end{pmatrix} \quad X_-:=\begin{pmatrix} 0 & 0 \\ 1 & 0  \end{pmatrix} \quad X_+:=\begin{pmatrix} 0 & 1 \\ 0 & 0  \end{pmatrix}\] and the associated Lie derivatives to these basis elements are~\cite[Chapter~IV, \S 4]{Lan85} \begin{align*} &\mathfrak{L}_H= -2y \sin 2\theta \frac{\partial}{\partial x}+2y \cos 2\theta \frac{\partial}{\partial y} + \sin 2 \theta \frac{\partial}{\partial \theta} \\
&\mathfrak{L}_{X_+}= y \cos 2\theta \frac{\partial}{\partial x}+y \sin 2\theta \frac{\partial}{\partial y} + \sin^2 \theta \frac{\partial}{\partial \theta}\\
&\mathfrak{L}_{X_-}= y \cos 2\theta \frac{\partial}{\partial x}+y \sin 2\theta \frac{\partial}{\partial y} - \cos^2 \theta \frac{\partial}{\partial \theta}.
  \end{align*}
  
 As $f$ is constant in $x$ and $\frac{\partial^i} {\partial y^i} \left(\ind_{[T, \infty), \eta}(\mathfrak{m}_{T/B_1}(y))\right)$ has support in $[B_1-|\eta|, B_1+|\eta|]$ for every $i \in \NN$, any term of $\mathfrak{L}_{X_l} \cdots \mathfrak{L}_{X_1} f$ having $y$ as a factor must have support in $[B_1-|\eta|, B_1+|\eta|]$ for the $y$-variable by elementary calculus. Likewise, any term not having $y$ as a factor must have support in $[B_1-|\eta|, \infty)$ for the $y$-variable.  Here, for $1\leq j \leq l$, $X_j$ is one of the standard basis elements of $\mathfrak{g}$.  Recalling that $~\wrt \mu = y^{-2}~\wrt x~\wrt y~\wrt \theta$, we have that $\mathfrak{L}_{X_l} \cdots \mathfrak{L}_{X_1} f \in L^2(\Gamma \backslash G)$.  Consequently, $\|f\|_{W_k} < \infty$.

  Also note that, by elementary calculus, every term of $\mathfrak{L}_{X_l} \cdots \mathfrak{L}_{X_1} f$ is of the form $g(\theta) y^i \frac{\partial^{i+j} f}{\partial y^i \partial \theta^j}$ for some $(i,j) \in (\NN \cup \{0\}) \times (\NN \cup \{0\})$ and some $C^\infty$-function $g(\theta)$.

We now estimate $\|f\|_{W_k}$.  By the above, we need only consider a finite number of monomials of the form $\mathfrak{L}_{X_l} \cdots \mathfrak{L}_{X_1} f$ where $1 \leq l \leq k$, and, thus, only a finite number of terms of the form $g(\theta) y^i \frac{\partial^{i+j} f}{\partial y^i \partial \theta^j}$ where $1\leq i+j \leq k$.  By the above, terms for which $i \geq 1$ have support in $[B_1-|\eta|, B_1+|\eta|]$.  Applying Young's inequality (as in the proof of Lemma~\ref{lemm:LipschitzConstChi}) and the chain rule, we have $\|\frac{\partial^i} {\partial y^i} \left(\ind_{[T, \infty), \eta}(\mathfrak{m}_{T/B_1}(y))\right)\|_\infty \leq O(T^i B_1^{-i} |\eta|^{-(i+1)})$ and, consequently, the contribution to $\|f\|_{W_k}$ from these terms are bounded by $O(T^i B_1^{-3/2}|\eta|^{-i}) \leq O(T^k B_1^{-3/2}|\eta|^{-k})$.  (Note that, here, $(B_1+|\eta|)^{i-1/2} - (B_1-|\eta|)^{i-1/2}$ contributes a factor of $B_1^{i-3/2}|\eta|$.)  The contribution from the terms for which $i=0$ are bounded by $O(B_1^{-1})$.  Consequently, $\|f\|_{W_k} \leq O(T^kB_1^{-1}|\eta|^{-k})$.  Here, all of the implied constants depend on $h$.  \end{proof}

Our variant of~\cite[Theorem~2]{St} is the following:
\begin{lemm} \label{lemm:HoroDistTanBun}  Let $j \in \{1, \cdots, q\}$, $\delta>0$ and fix $A > B_0> 1$.  Then \[\frac 1 {\beta - \alpha}\int_\alpha^\beta \phi^{(\kappa_j)}_{A,0}(x+iy, 0)~\wrt x \rightarrow \frac 1{\mu(\Gamma \backslash G)} \int_{\Gamma \backslash G}
\phi^{(\kappa_j)}_{A,0}(p)~\wrt \mu(p)   
 = \frac{1}{A\mu(\Gamma \backslash G)} \int_0^{2\pi}\int_0^1 h(x, \theta)~\wrt x~\wrt\theta\] uniformly as $y \rightarrow 0$ so long as $\beta - \alpha$ remains bigger than $y^{1/2 - \delta}$.
 
\end{lemm}
\begin{proof}  Let $A - B_0>\varepsilon>0$ and $\phi_A:=\phi^{(\kappa_j)}_{A,0}$.  Using the smooth Urysohn lemma, define $C^\infty$-functions $\ind^+: (0, \infty) \rightarrow [0,1]$ and $\ind^-: (0, \infty)  \rightarrow [0,1]$ by \begin{align*}  \ind^+(y) &:= \ind_{[A, \infty)}^{\varepsilon, +}(y) := \begin{cases} 1 & \text{ if } y \in [A, \infty) \\ 0 & \text{ if } y \in (0, A- \varepsilon]  \end{cases} \\ \ind^-(y) &:= \ind_{[A, \infty)}^{\varepsilon, -}(y) := \begin{cases} 1 & \text{ if } y \in [A+\varepsilon, \infty) \\ 0 & \text{ if } y \in (0, A]  \end{cases}.
 \end{align*}  (Here, we are {\em not} giving an explicit formula for $\ind^+ \vert_{(A- \varepsilon, A)}$ or for $ \ind^- \vert_{(A, A+ \varepsilon)}$.)  These functions, furthermore, define the $C^\infty$-functions \begin{align*} f^+(z, \theta) &=  \ind^+(y) h(x, \theta)\quad \textrm{ and } \quad \phi^+(z, \theta) = \sum_{\gamma \in \Gamma_\infty \backslash \sigma_j^{-1}\Gamma} f^+(\gamma(z, \theta)) \\
  f^-(z, \theta) &=  \ind^-(y) h(x, \theta)\quad \textrm{ and } \quad \phi^-(z, \theta) = \sum_{\gamma \in \Gamma_\infty \backslash \sigma_j^{-1}\Gamma} f^-(\gamma(z, \theta)).\end{align*}  Recall (\ref{eqn:TestFctDefn},~\ref{eqnDefnFT0}) that \[\phi_A(x+iy, 0)= \sum_{\gamma \in \Gamma_\infty \backslash \sigma_j^{-1}\Gamma} f_{A,0}(\gamma(z, \theta)),\] and, thus, we have that the inequalities $\phi^- \leq \phi_A \leq \phi^+$ hold pointwise.   Let \[M := \max \{|h(x, \theta)| : x \in [0,1] \textrm { and } \theta \in [0, 2\pi]\}.\]

 Let us first give the proof for the case $h\geq 0$.  Let $\epsilon>0$. Choose $\varepsilon$ such that $\frac{2 \pi M \varepsilon}{A - \varepsilon}< \epsilon/2$.  The inequalities imply that \begin{align*} \frac 1 {\beta - \alpha}\int_\alpha^\beta \phi_A(x+iy, 0)~\wrt x &\leq \frac 1 {\beta - \alpha}\int_\alpha^\beta \phi^+(x+iy, 0)~\wrt x  \\ &\leq \frac{1}{A\mu(\Gamma \backslash G)} \left(\int_0^{2\pi}\int_0^1 h(x, \theta)~\wrt x~\wrt\theta + \frac{2 \pi M \varepsilon}{A - \varepsilon}\right) + \frac \epsilon 2 \\
 &\leq \frac{1}{A\mu(\Gamma \backslash G)} \left(\int_0^{2\pi}\int_0^1 h(x, \theta)~\wrt x~\wrt\theta \right)+ \epsilon 
  \end{align*} for $y$ small enough and all $\alpha \leq \beta$ such that $\beta - \alpha\geq y^{1/2 -\delta}$.  Here the second inequality follows from~\cite[Theorem~2]{St}.  Using the analogous proof for $\phi^-$, we obtain \begin{align*} \frac 1 {\beta - \alpha}\int_\alpha^\beta \phi_A(x+iy, 0)~\wrt x  \geq \frac{1}{A\mu(\Gamma \backslash G)} \left(\int_0^{2\pi}\int_0^1 h(x, \theta)~\wrt x~\wrt\theta \right)- \epsilon 
  \end{align*} for $y$ small enough and all $\alpha \leq \beta$ such that $\beta - \alpha\geq y^{1/2 -\delta}$.  Consequently, we have that \[\left| \frac 1 {\beta - \alpha}\int_\alpha^\beta \phi_A(x+iy, 0)~\wrt x -  \frac{1}{A\mu(\Gamma \backslash G)} \left(\int_0^{2\pi}\int_0^1 h(x, \theta)~\wrt x~\wrt\theta \right)\right| < \epsilon\] for $y$ small enough and all $\alpha \leq \beta$ such that $\beta - \alpha\geq y^{1/2 -\delta}$, which proves the desired result when $h \geq 0$.
  
 Since $h + M \geq 0$, applying the above proof with $h$ replaced by $h +M$ yields the desired result for $h+M$.  Applying the above proof with $h$ replaced by $M$ yields the desired result for $M$.  Subtracting these two results yields the desired result for general $h$ and proves the desired result in all cases.
\end{proof}

\section{Proof of Theorem~\ref{thm:TanBunRelEqui}}\label{sec:ProofthmTanBunRelEqui}  In this section, we prove Theorem~\ref{thm:TanBunRelEqui}, which is our non-effective main result and which allows a greater range for the rate of growth of $T$ versus the decay of $Ty$ than our effective main result.  The idea of the proof is as follows.  We will approximate integrating from $\alpha$ to $\beta$ via $C^\infty$-functions constructed from a well-known mollifier.  Using the Fourier series of these approximations, we apply Proposition~\ref{prop:MovingToFixed2} to the low modes of these Fourier series.  Using a classical result for the uniform convergence of the Fourier series for smooth functions, we show the high modes are negligible.  This allows us to apply horocycle equidistribution for fixed functions, namely Lemma~\ref{lemm:HoroDistTanBun}, to obtain the desired result.

\begin{proof}[Proof of Theorem~\ref{thm:TanBunRelEqui}]  Let $\phi_T:=\phi^{(\kappa_j)}_{T,0}$, $\phi^0:= \phi^{0,(\kappa_j)}_0$, and $\alpha < \beta$.  Without loss of generality, we may assume that $0\leq \alpha$, $\beta \leq 1$, and $1/2\geq\delta >0$.  Otherwise, we can break the integral into a finite number of pieces and use the periodicity in $x$.  Let us first assume that either $0\leq \alpha< \beta <1$ or $0< \alpha <\beta\leq1$.  The final case of $\alpha=0$ and $\beta=1$ is a simplification and will be proved at the end.  

Let $\frac 1 2 (\beta-\alpha) > \varepsilon>0$ and \[M := \max \{|h(x, \theta)| : x \in [0,1] \textrm { and } \theta \in [0, 2\pi]\}.\]  Define the indicator function $\chi: \RR \rightarrow [0,1]$ by \[\chi(x):=\chi_{[\alpha, \beta]}(x):=\begin{cases} 1 & \textrm{ if } x \in [\alpha, \beta] \\ 0 & \textrm{ if }  x \notin [\alpha, \beta] \end{cases}.\]  Because we require explicit Lipschitz constants, we will approximate $\chi$ via convolutions with an explicit $C^\infty$-function, namely the well-known mollifier (see~\cite[Example~9.23]{Tah15} for example) $\rho: \RR \rightarrow \RR$ defined by \begin{align}\label{eqn:MolliferRhoDefn}
 \rho(x) := \begin{cases} \ell^{-1}e^{- \frac 1 {1 - x^2}} &\textrm{ if } |x| < 1 \\ 0 & \textrm{ if } |x| \geq 1 \end{cases} \quad \quad \rho_\varepsilon(x) := \frac 1 \varepsilon \rho\left(\frac x \varepsilon\right) \end{align} where \[\ell := \int_{-1}^1 e^{- \frac 1 {1 - x^2}} ~\wrt x.\]  Define $C^\infty$-functions $\psi^+: \RR \rightarrow [0,1]$ and $\psi^-: \RR  \rightarrow [0,1]$ by the convolutions \begin{align*} \psi^+(x) := \psi_{[\alpha, \beta]}^{\varepsilon, +}(x):=\int_{-\infty}^\infty \chi_{[-r_+, r_+]}(x-y) \rho_{\varepsilon/2}(y)~\wrt y \\
\psi^-(x) :=\psi_{[\alpha, \beta]}^{\varepsilon, -}(x):= \int_{-\infty}^\infty \chi_{[-r_-, r_-]}(x-y) \rho_{\varepsilon/2}(y)~\wrt y
 \end{align*} where $r_+:= \frac{\beta-\alpha+\varepsilon} 2$ and $r_-:= \frac{\beta-\alpha-\varepsilon} 2$.  Translating these functions, we obtain our desired upper and lower approximations: \begin{align*}  \chi^+(x) := \chi_{[\alpha, \beta]}^{\varepsilon, +}(x) :=\psi^+\left(x - \frac{\alpha + \beta}{2}\right) \quad \quad \chi^-(x) := \chi_{[\alpha, \beta]}^{\varepsilon, -}(x) := \psi^-\left(x - \frac{\alpha + \beta}{2}\right) \end{align*}   Note that here $ \chi^+$ and $\chi^-$ are both $C^\infty$-functions with $0\leq \chi^+ \leq 1$ and $0\leq \chi^- \leq 1$ such that \begin{align*} &\chi^+ \vert_{[\alpha, \beta]} =1, \quad \quad  \chi^+ \vert_{(-\infty, \alpha - \varepsilon] \cup [\beta + \varepsilon, \infty)  } =0 \\
 &\chi^- \vert_{ [\alpha+ \varepsilon, \beta - \varepsilon] } =1, \quad \quad \chi^- \vert_{(-\infty, \alpha] \cup [\beta, \infty) } =0.
  \end{align*}
  
Now, as $\rho$ is smooth and compactly supported, the bounds $M_\rho(n):= \max \left\{\left|\frac{\wrt{}^n \rho}{\wrt x^n} (x)\right| : x \in \RR\right\}$ are finite for every $n \in \NN$.  For concision, let us define $\frac{\wrt{}^0\chi^+}{\wrt x^0}$ to be $\chi^+$ and $\frac{\wrt{}^0\chi^-}{\wrt x^0}$ to be $\chi^-$.
  
\begin{lemm}\label{lemm:LipschitzConstChi}  Let $n \in \NN \cup \{0\}$.  A Lipschitz constant for $\frac{\wrt{}^n\chi^+}{\wrt x^n}$ is $\left(\frac 2 \varepsilon\right)^{n+2}(\beta - \alpha + \varepsilon)M_\rho(n+1)$ and a Lipschitz constant for $\frac{\wrt{}^n\chi^-}{\wrt x^n}$ is $\left(\frac 2 \varepsilon\right)^{n+2}(\beta - \alpha - \varepsilon)M_\rho(n+1)$.
\end{lemm}

\begin{proof}
By the mean value theorem, a bound for the derivative will be a Lipschitz constant for the function.  By induction on the chain rule, we have \[\left|\frac{\wrt{}^{n+1} \rho_{\varepsilon/2}}{\wrt x^{n+1}} (x)\right| \leq \left(\frac{2} {\varepsilon}\right)^{n+2}M_\rho(n+1)\] for all $x \in \RR$.  Applying Young's inequality gives \[\left\|\frac{\wrt{}^{n+1} \chi^+}{\wrt x^{n+1}}\right\|_\infty \leq \left\|\chi_{[-r_+, r_+]}\right\|_1 \left\|\frac{\wrt{}^{n+1} \rho_{\varepsilon/2}}{\wrt x^{n+1}} \right\|_\infty \leq \left(\frac{2} {\varepsilon}\right)^{n+2}(\beta - \alpha + \varepsilon)M_\rho(n+1)\] and the desired result.  The proof for $\chi^-$ is analogous.
\end{proof}

\begin{coro} Let $n \in \NN \cup \{0\}$.  The functions \begin{align*}
\omega_{+,n}(t):= & \left(\left(\frac 2 \varepsilon\right)^{n+2}(\beta - \alpha + \varepsilon)M_\rho(n+1)\right) t \\ \omega_{-,n}(t):= &\left(\left(\frac 2 \varepsilon\right)^{n+2}(\beta - \alpha - \varepsilon)M_\rho(n+1)\right)t  \end{align*} are moduli of continuity for $\frac{\wrt{}^n\chi^+}{\wrt x^n}$ and $\frac{\wrt{}^n\chi^-}{\wrt x^n}$, respectively.
\end{coro}

\begin{proof}
 This is immediate from the definitions.
\end{proof}

\begin{lemm}\label{lemm:bndFourierSmoothCutoff} Let $n \in \NN \cup \{0\}$ and $m \in \ZZ$.  We have that \begin{align*} \left|\widehat{\frac{\wrt{}^n\chi^+}{\wrt x^n}}(m)\right| &\leq 2(\beta-\alpha+\varepsilon)\left(\frac 2 \varepsilon\right)^{n} M_\rho(n)\\
 \left|\widehat{\frac{\wrt{}^n\chi^-}{\wrt x^n}}(m)\right| &\leq 2(\beta-\alpha-\varepsilon)\left(\frac 2 \varepsilon\right)^{n} M_\rho(n)\end{align*} for all small enough $\varepsilon$.
 
\end{lemm}

\begin{proof}
 Since the supports of $\frac{\wrt{}^n\chi^+}{\wrt x^n}$ and $\frac{\wrt{}^n\chi^-}{\wrt x^n}$ are contained in $[0,1]$ (perhaps after applying the periodicity in $x$ and, because $\varepsilon$ is small enough, there are no overlaps), we can replace the Fourier transform over the circle by that over the line.  By elementary properties of convolutions and Fourier analysis, it suffices to give a bound for \begin{align*}\left|\int_{-\infty}^\infty\frac{\wrt{}^{n} \rho_{\varepsilon/2}}{\wrt x^{n}} (x) e(-mx)~\wrt x\right| &= \left|{\left(\frac 2 \varepsilon\right)^{n+1}}\int_{-\infty}^\infty\frac{\wrt{}^{n} \rho}{\wrt x^{n}} \left(\frac{2x}{\varepsilon}\right) e(-mx)~\wrt x\right|\\
 &= \left|{\left(\frac 2 \varepsilon\right)^{n+1}}\int_{-\varepsilon/2}^{\varepsilon/2}\frac{\wrt{}^{n} \rho}{\wrt x^{n}} \left(\frac{2x}{\varepsilon}\right) e(-mx)~\wrt x\right| \\
 &= \left|{\left(\frac 2 \varepsilon\right)^{n}}\int_{-1}^{1}\frac{\wrt{}^{n} \rho}{\wrt x^{n}} \left(x\right) e(-\varepsilon mx/2)~\wrt x\right|\\
 &\leq 2\left(\frac 2 \varepsilon\right)^{n} M_\rho(n).\end{align*}  The first equality follows from induction on the chain rule, the second equality follows from the fact that the supports of all of the derivatives of $\rho$ are contained in the support of $\rho$, namely the closed interval $[-1,1]$, and the third equality follows from changing variables $\frac{2x} {\varepsilon} \mapsto x$.  The desired result now follows. \end{proof}

 By smoothness and periodicity, we have \begin{align*}\chi^+(x)   = \sum_{m \in \ZZ} \widehat{\chi^+}(m) e(m x) \quad \quad \chi^-(x)   = \sum_{m \in \ZZ} \widehat{\chi^-}(m) e(m x) \end{align*} where $\widehat{\chi^+}(m)$ and $\widehat{\chi^-}(m)$ are the Fourier transforms on $\RR / \ZZ$ of $\chi^+(x)$ and $\chi^-(x)$, respectively.  Moreover, the convergence of each sum to each function is uniform, and, we have, by a classical theorem due to Jackson~\cite[Page~21]{Jac30} \begin{align*} &
 \left|\chi^+(x) - \sum_{m =-N}^N \widehat{\chi^+}(m) e(m x) \right| \leq  K_0 \omega_{+,n}(2 \pi/N) \frac{\log N}{N^n}= \frac { K_1(\beta - \alpha + \varepsilon)\log N} {\varepsilon^{n+2}N^{n+1}} \\ &\left|\chi^-(x) - \sum_{m =-N}^N \widehat{\chi^-}(m) e(m x) \right| \leq   K_0 \omega_{-,n}(2 \pi/N) \frac{\log N}{N^n}=\frac { K_1(\beta - \alpha - \varepsilon)\log N} {\varepsilon^{n+2}N^{n+1}}\end{align*} where $K_0>0$ is a fixed constant independent of $x$, $N$, $n$, $\varepsilon$, $\rho$, $\alpha$, and $\beta$ and where $K_1:= \pi 2^{n+3}  K_0 M_\rho(n+1)$.  Note that both bounds are uniform in $x$.
 
 Let us first give the proof for the case $h\geq 0$.  In particular, this implies that $\phi_T \geq 0$ pointwise, from which it follows that $\phi_T \chi^- \leq \phi_T\chi \leq  \phi_T\chi^+$ pointwise.   Let \[n:= \frac 2 \delta \quad \quad \varepsilon:= \frac 1 8 \max\left\{T^{-1/6},\left(\frac{Ty}{B_1}\right)^{1/2}\right\}\left(\frac{Ty}{B_1}\right)^{-\delta/2} \quad \quad \epsilon>0 \quad \quad N:=\left\lfloor \left(\frac{Ty}{B_1}\right)^{-1/2 + \delta/4}\right\rfloor.\]  We have that \begin{align}\label{eqn:UnitTanBunUppBnd}
 \frac T {\beta - \alpha} & \int_\alpha^\beta \phi_T(x+iy, 0)~\wrt x = \frac T {\beta - \alpha} \int_0^1 \phi_T(x+iy, 0)\chi(x)~\wrt x \\\nonumber&\leq  \frac T {\beta - \alpha} \sum_{m =-N}^N \widehat{\chi^+}(m)\int_0^1 \phi_T(x+iy, 0)e(mx)~\wrt x  \\\nonumber &\quad \quad +   \frac T {\beta - \alpha}\frac { K_1(\beta - \alpha + \varepsilon)\log N} {\varepsilon^{n+2}N^{n+1}} \int_0^1 \phi_T(x+iy, 0)~\wrt x.\end{align}
 
 Applying Proposition~\ref{prop:MovingToFixed2}, we obtain, as $T \rightarrow \infty$ and $Ty \rightarrow 0$, \begin{align} \label{eqn:ConstModeTransfer} TK_1 &\int_0^1 \phi_T(x+iy, 0)~\wrt x \\\nonumber &= B_1K_1\left(1 + O\left(\frac{Ty}{B_1}\right)^{\delta} \right) \int_0^1 \phi^0\left(x+i\frac{T}{B_1}y, 0\right)~\wrt x +O\left(\frac {K_1}{\sqrt{T}}\right).
  \end{align}  Now, by Lemma~\ref{lemm:HoroDistTanBun} , we have that, as $Ty \rightarrow 0$, \begin{align*}
B_1 K_1 &\int_0^1 \phi^0\left(x+i\frac{T}{B_1}y, 0\right)~\wrt x \rightarrow  \frac {K_1} {\mu(\Gamma \backslash G)} \int_0^{2\pi}\int_0^1 h(x, \theta)~\wrt x~\wrt\theta \leq2 \pi M \frac {K_1} {\mu(\Gamma \backslash G)}.\end{align*}  Moreover, as $\beta -\alpha \geq \max\left\{T^{-1/6},\left(\frac{Ty}{B_1}\right)^{1/2}\right\}\left(\frac{Ty}{B_1}\right)^{-\delta}$, we have \begin{align}\label{eqn:EpsilonLessThanWhole}
0 \leq \frac{ \varepsilon}  {\beta - \alpha}\leq \frac 1 8 \left(\frac{Ty}{B_1}\right)^{\delta/2},
\end{align} and we also have that  \begin{align*} &0\leq \frac 1{\varepsilon^{n+2}N^{n+1-\delta/2}}\leq 2^{2/\delta+1 - \delta/2} 8^{2/\delta+2} \left(\frac{Ty}{B_1}\right)^{\delta/2 + \delta^2/8}.  \end{align*}  Consequently, for all $T$ large enough and all $Ty$ small enough, we have that \begin{align} \label{eqn:ConstModeBnd} \frac T {\beta - \alpha}\frac { K_1(\beta - \alpha + \varepsilon)\log N} {\varepsilon^{n+2}N^{n+1}} \int_0^1 \phi_T(x+iy, 0)~\wrt x \leq \frac{\epsilon} {4}.
  \end{align}

 Another application of Proposition~\ref{prop:MovingToFixed2} yields, as $T \rightarrow \infty$ and $Ty \rightarrow 0$, \begin{align*}  &\frac T {\beta - \alpha} \sum_{m =-N}^N \widehat{\chi^+}(m)\int_0^1 \phi_T(x+iy, 0)e(mx)~\wrt x \\ & \ = \frac{B_1}{\beta-\alpha} \left(1 + O\left(\frac{Ty}{B_1}\right)^{\delta} \right)\sum_{m =-N}^N \widehat{\chi^+}(m)\int_0^1 \phi^0\left(x+i\frac{T}{B_1}y, 0\right)e(mx)~\wrt x +\sum_{m =-N}^NO\left(\frac {\widehat{\chi^+}(m)}{(\beta-\alpha)\sqrt{T}}\right)  \\
 & \ \leq \frac{B_1}{\beta-\alpha} \left(1 + O\left(\frac{Ty}{B_1}\right)^{\delta} \right)\int_0^1 \phi^0\left(x+i\frac{T}{B_1}y, 0\right)\chi^+(x)~\wrt x \\
& \quad + \frac{B_1}{\beta-\alpha}\left(1 + O\left(\frac{Ty}{B_1}\right)^{\delta} \right) \frac { K_1(\beta - \alpha + \varepsilon)\log N} {\varepsilon^{n+2}N^{n+1}}\int_0^1 \phi^0\left(x+i\frac{T}{B_1}y, 0\right)~\wrt x  +\sum_{m =-N}^NO\left(\frac {\left|\widehat{\chi^+}(m)\right|}{(\beta-\alpha)\sqrt{T}}\right)\end{align*}   Note, as $h \geq 0$, we have that $\phi^0\geq0$ pointwise.  By (\ref{eqn:ConstModeTransfer}, \ref{eqn:ConstModeBnd}), we have, for all $T$ large enough and all $Ty$ small enough, that \begin{align}\label{eqn:ConstModeBnd2} \frac{B_1}{\beta-\alpha}  \left(1 + O\left(\frac{Ty}{B_1}\right)^{\delta} \right)\frac { K_1(\beta - \alpha + \varepsilon)\log N} {\varepsilon^{n+2}N^{n+1}}\int_0^1 \phi^0\left(x+i\frac{T}{B_1}y, 0\right)~\wrt x  \leq \frac{\epsilon} {4}.
  \end{align} 
  
  Now, by Lemma~\ref{lemm:bndFourierSmoothCutoff}, we have that \begin{align}\label{eqn:ErrorTermSqrt}\sum_{m =-N}^NO\left(\frac {\left|\widehat{\chi^+}(m)\right|}{(\beta-\alpha)\sqrt{T}}\right) &\leq O\left(\frac {1}{(\beta-\alpha)\sqrt{T}}\right) \left(\sum_{m \in \ZZ \backslash \{0\}}\frac{\left|\widehat{\frac{\wrt{}^2\chi^+}{\wrt x^2}}(m)\right|}{4\pi^2m^2} + \int_0^1 \chi^+(x)~\wrt x \right) \\\nonumber &= O\left(\frac {1}{\varepsilon^2(\beta-\alpha)\sqrt{T}}\right) \leq O\left(\left(\frac{Ty}{B_1}\right)^{2\delta} \right).
  \end{align}  Note that the implied constant depends on $h$, $\kappa_1$, and $\kappa_j$.  Consequently, for all $T$ large enough and all $Ty$ small enough, we have that \begin{align*}\sum_{m =-N}^NO\left(\frac {\left|\widehat{\chi^+}(m)\right|}{(\beta-\alpha)\sqrt{T}}\right) \leq \frac{\epsilon} {4}.\end{align*}
  
  Finally, to obtain an upper bound for the left-hand side of (\ref{eqn:UnitTanBunUppBnd}), we must give an upper bound for \[D:=\frac{B_1}{\beta-\alpha} \left(1 + O\left(\frac{Ty}{B_1}\right)^{\delta} \right)\int_0^1 \phi^0\left(x+i\frac{T}{B_1}y, 0\right)\chi^+(x) ~\wrt x.\]  Since $\phi^0\geq0$ pointwise, we have that \begin{align}\label{eqn:DUpperBnd} D \leq  \left(1 + \frac{ 2 \varepsilon}{\beta-\alpha}\right)\frac{B_1}{\beta-\alpha+ 2 \varepsilon} \left(1 + O\left(\frac{Ty}{B_1}\right)^{\delta} \right)\int_{\alpha - \varepsilon}^{\beta + \varepsilon} \phi^0\left(x+i\frac{T}{B_1}y, 0\right)~\wrt x.
 \end{align}
 Applying (\ref{eqn:EpsilonLessThanWhole}) and Lemma~\ref{lemm:HoroDistTanBun}, we have, for all $T$ large enough and $Ty$ small enough, that\begin{align} \label{eqn:DUpperBnd2}  D \leq    \frac {1} {\mu(\Gamma \backslash G)} \int_0^{2\pi}\int_0^1 h(x, \theta)~\wrt x~\wrt\theta + \frac{\epsilon} 4.\end{align}  Note that Lemma~\ref{lemm:HoroDistTanBun} applies because $\beta-\alpha+ 2 \varepsilon \geq \beta -\alpha \geq \left(\frac{Ty}{B_1}\right)^{1/2 -\delta}$ holds.  Consequently, we can conclude that \begin{align*} \frac T {\beta - \alpha} & \int_\alpha^\beta \phi_T(x+iy, 0)~\wrt x \leq  \frac {1} {\mu(\Gamma \backslash G)} \int_0^{2\pi}\int_0^1 h(x, \theta)~\wrt x~\wrt\theta +\epsilon.
  \end{align*}
   
 Using the analogous proof for $\chi^-$, we obtain that\begin{align*} \frac T {\beta - \alpha} & \int_\alpha^\beta \phi_T(x+iy, 0)~\wrt x \geq  \frac {1} {\mu(\Gamma \backslash G)} \int_0^{2\pi}\int_0^1 h(x, \theta)~\wrt x~\wrt\theta -\epsilon,
  \end{align*} where we are able to apply Lemma~\ref{lemm:HoroDistTanBun} to compute the main term in the analogous proof because \begin{align*}
\beta - \alpha - 2 \varepsilon &\geq  \max\left\{T^{-1/6},\left(\frac{Ty}{B_1}\right)^{1/2}\right\}\left(\frac{Ty}{B_1}\right)^{-\delta} - \frac 1 4 \max\left\{T^{-1/6},\left(\frac{Ty}{B_1}\right)^{1/2}\right\}\left(\frac{Ty}{B_1}\right)^{-\delta/2} \\ &\geq \ \max\left\{T^{-1/6},\left(\frac{Ty}{B_1}\right)^{1/2}\right\}\left(\frac{Ty}{B_1}\right)^{-\delta/2}  \end{align*} for $Ty$ small enough.  This yields the desired result for the case $h\geq0$.
  
Since $h + M \geq 0$, applying the above proof with $h$ replaced by $h +M$ yields the desired result for $h+M$.  Applying the above proof with $h$ replaced by $M$ yields the desired result for $M$.  Subtracting these two results yields the desired result for general $h$.

The final case to consider is when $\alpha=0$ and $\beta=1$ (namely, the case of closed horocycles).  This final case follows by applying Proposition~\ref{prop:MovingToFixed2} (with $m=0$) and Lemma~\ref{lemm:HoroDistTanBun}.  This proves the desired result in all cases.
\end{proof}

\section{Proof of Theorem~\ref{thm:EffectTanBunRelEqui}}\label{sec:ProofthmEffectTanBunRelEqui}  In this section, we prove Theorem~\ref{thm:EffectTanBunRelEqui}, which is our effective main result.  The proof is analogous to that of Theorem~\ref{thm:TanBunRelEqui} except that we replace Lemma~\ref{lemm:HoroDistTanBun} with its effective version, Lemma~\ref{lemm:EffectHoroDistTanBun}, and keep track of the error terms.  For $\eta =0$, there is an additional step of obtaining the correct main term.  Before we give the details, let us first note that \[\max\left(T^4 \left(\frac{Ty}{B_1}\right)^{1/2}, T^{-1/6}\right)\left(\frac{Ty}{B_1}\right)^{-\delta}\] can decay provided the relative rate of decay of $y$ to the growth of $T$ is constrained.  For example, if we substitute $y =T^{-10}$ and $\delta = \frac 1 {60}$ into this expression, then we obtain $T^{-1/60} B_1^{1/60}$, which decays as $T \rightarrow \infty$.  Also note that, as this expression is larger than or equal to the analogous expression for Theorem~\ref{thm:TanBunRelEqui}, there is no hinderance to applying the proof of Theorem~\ref{thm:TanBunRelEqui} with the noted changes.

\begin{proof}[Proof of Theorem~\ref{thm:EffectTanBunRelEqui}] Let $0< |\widetilde{\eta}| \leq \min\left(\frac{B_1-B_0}{2}, \frac 1 4 \right)$, $\phi_T:=\phi^{(\kappa_j)}_{T,\eta},$ and $\phi^0:= \phi^{0,(\kappa_j)}_{\widetilde{\eta}}$.  We follow the proof of Theorem~\ref{thm:TanBunRelEqui}.  Let \begin{align*}
Q &:= \frac {T}{\mu(\Gamma \backslash G)} \int_0^{2\pi} \int_0^1 \int_0^\infty \ind_{[T, \infty), \widetilde{\eta}}(y) h(x, \theta) \frac{\wrt y \wrt x \wrt \theta}{y^2} 
\\ R &:= R(\widetilde{\eta}):=O\left(T^4 |\widetilde{\eta}|^{-4}\right) \left( \left(\frac{Ty}{B_1}\right)^{1/2} \log^2\left(\frac{B_1}{Ty}\right) + \left(\frac{Ty}{B_1}\right)^{1-s_1}\right)
\\ \widetilde{R}&:=\widetilde{R}(\widetilde{\eta}):=O\left(T^4 |\widetilde{\eta}|^{-4}\right) \\ & \quad\quad\quad\quad\times \left( \left(\frac{Ty}{(\beta-\alpha)^2B_1}\right)^{1/2} \log^2\left(\frac{(\beta-\alpha)B_1}{Ty}\right) + \left(\frac{Ty}{(\beta-\alpha)^2B_1}\right)^{1-s'_1} + \left(\frac{Ty}{(\beta-\alpha)B_1}\right)^{1-s_1}\right).\end{align*}  Note that $|Q| \leq \frac{8\pi M}{3\mu(\Gamma \backslash G)}$.  By applying Lemma~\ref{lemm:EffectHoroDistTanBun} in place of Lemma~\ref{lemm:HoroDistTanBun}, we replace (\ref{eqn:ConstModeBnd}) with the following:  \begin{align}\label{eqn:EffectConstModeBnd} \frac T {\beta - \alpha}&\frac { K_1(\beta - \alpha + \varepsilon)\log N} {\varepsilon^{n+2}N^{n+1}} \int_0^1 \phi_T(x+iy, 0)~\wrt x  
\\\nonumber &\leq 2^{2/\delta} 8^{2/\delta}K_1\left(\frac{Ty}{B_1}\right)^{\delta/2 + \delta^2/8} \left(1+ \frac 1 8 \left(\frac{Ty} {B_1}\right)^{\delta/2}\right)\left(\left(1+O\left(\frac{Ty}{B_1}\right)^{\delta} \right)(Q+R) +O\left(\frac{1}{\sqrt{T}} \right)\right).
  \end{align}  Note that (\ref{eqn:EffectConstModeBnd}) holds when $\frac{\log N}{N^{\delta/2}} \leq 128$, a condition which, for the given $\delta>0,$ we can ensure for all $Ty$ small enough.  Also, the left-hand side of (\ref{eqn:ConstModeBnd2}) is bounded by the right-hand side of (\ref{eqn:EffectConstModeBnd}).  Moreover, (\ref{eqn:ErrorTermSqrt}) holds, giving an additional error term of $O\left(\frac{Ty}{B_1}\right)^{2\delta}$.

  Finally, we compute $D$, which will yield the main term (and additional error terms).  By applying Lemma~\ref{lemm:EffectHoroDistTanBun} in place of Lemma~\ref{lemm:HoroDistTanBun} to (\ref{eqn:DUpperBnd}), we obtain the analog of (\ref{eqn:DUpperBnd2}):  \begin{align*}  D \leq \left(1 + O\left(\frac{Ty}{B_1} \right)^{\delta/2}\right) \left(Q + \widetilde{R} \right).
  \end{align*}  Note that $\varepsilon \leq \frac 1 8 (\beta-\alpha)$.  Thus, the main term is $Q$ and the error terms that are significant coming from this expression are $O(|Q|)\left(\frac{Ty}{B_1} \right)^{\delta/2} = O\left(\frac{Ty}{B_1} \right)^{\delta/2}$ and $\widetilde{R}$.\footnote{Theorem~\ref{thm:EffectTanBunRelEqui} provides a meaningful result only when the term $\widetilde{R}$ decays.}  Here, the implied constant depends on $h$, $\kappa_1,$ and $\kappa_j$.  Note that the error term coming from (\ref{eqn:ErrorTermSqrt}) is negligible compared to these terms.  Similarly, for $Ty$ small enough, the bound coming from the right-hand side of (\ref{eqn:EffectConstModeBnd}) is negligible. 
  
  Consequently, we have that \[ \frac T {\beta - \alpha} \int_\alpha^\beta \phi_T(x+iy, 0)~\wrt x \leq Q	+  \widetilde{R} + O\left(\frac{Ty}{B_1} \right)^{\delta/2}\] for $T \rightarrow \infty$ and $Ty \rightarrow 0$.  Giving the analogous proof for $\chi^-$, yields the reverse inequality and, thus, equality for $h \geq 0$:  \begin{align}\label{eqn:MainEffectEquality} \frac T {\beta - \alpha} \int_\alpha^\beta \phi_T(x+iy, 0)~\wrt x = Q	+  \widetilde{R} + O\left(\frac{Ty}{B_1} \right)^{\delta/2}.
  \end{align}  Here, the implied constant depends on $h$, $\kappa_1,$ and $\kappa_j$.  In the analogous way as in the proof of Theorem~\ref{thm:TanBunRelEqui}, we obtain (\ref{eqn:MainEffectEquality}) for general $h$.
  
Similar to the proof of Theorem~\ref{thm:TanBunRelEqui}, it remains to show the analog of (\ref{eqn:MainEffectEquality}) for the case of $\alpha=0$ and $\beta=1$.  By applying Lemma~\ref{lemm:EffectHoroDistTanBun} in place of Lemma~\ref{lemm:HoroDistTanBun} in the proof of Theorem~\ref{thm:TanBunRelEqui}, we obtain \begin{align}\label{eqn:SecondEffectEquality} T \int_0^1 \phi_T(x+iy, 0)~\wrt x = Q	+  R + O\left(\left(\frac{Ty}{B_1} \right)^{\delta} +\frac 1{\sqrt{T}}\right).
  \end{align}  Here, the implied constant depends on $h$, $\kappa_1,$ and $\kappa_j$.  Now, for $\eta \neq 0$, set $\widetilde{\eta} = \eta$ to obtain the desired result.  
  
 Finally, for $\eta =0$, we can pick any $0< |\widetilde{\eta}| \leq \min\left(\frac{B_1-B_0}{2}, \frac 1 4 \right)$.  Let us first assume that $\widetilde{\eta}>0$.  Note that \begin{align*}
Q &= \frac{T}{\mu(\Gamma\backslash G)}  \int_0^{2\pi} \int_0^1 \int_T^\infty  h(x, \theta) \frac{\wrt y~\wrt x~\wrt \theta}{y^2} \\ & \quad\quad + \frac{T}{\mu(\Gamma\backslash G)} \int_0^{2\pi} \int_0^1 \int_{T-\widetilde{\eta}}^T  \ind_{[T, \infty), \widetilde{\eta}}(y) h(x, \theta) \frac{\wrt y~\wrt x~\wrt \theta}{y^2}
\\ &= \frac{1}{\mu(\Gamma\backslash G)} \int_0^{2\pi} \int_0^1  h(x, \theta)~\wrt x~\wrt \theta + O\left(\frac{|\widetilde{\eta}|}{T} \right). 
  \end{align*}  The implied constant depends on $h$ and $\Gamma$.  A similar proof for $\widetilde{\eta}<0$ yields the same result.  This obtains the desired result and concludes the proof of the theorem.  \end{proof}
  
\section{Proof of Theorem~\ref{thm:EscNeighDepOnConst}}\label{sec:EscNeighDepOnConst}  This proof is a simplification of the proof of Theorem~\ref{thm:TanBunRelEqui}.  For the convenience of the reader, we now give the details.  

\begin{proof}[Proof of Theorem~\ref{thm:EscNeighDepOnConst}]  Let $\phi_T:=\phi^{(\kappa_j)}_{T,\eta}$, $\phi^0:= \phi^{0,(\kappa_j)}_{\widetilde{\eta}}$, and $\alpha < \beta$.  Without loss of generality, we may assume that $0\leq \alpha$ and $\beta \leq 1$.  Otherwise, we can break the integral into a finite number of pieces and use the periodicity in $x$.  If $\alpha=0$ and $\beta=1$, then apply Corollary~\ref{coro:MovingToFixed} to obtain the desired result.  Otherwise, let us assume that either $0\leq \alpha< \beta <1$ or $0< \alpha <\beta\leq1$.  Define $\chi$, $\chi^{\pm}$, $M$, and $M_\rho(n)$ as in the proof of Theorem~\ref{thm:TanBunRelEqui}.

Now set $\varepsilon:=\frac 1 4 T^{-3/4+\delta/2}$.  Let us first consider the case $h\geq0$.  We have, for all $T$ sufficiently large, that\begin{align*} \int_\alpha^\beta \phi_T &(x+iy, 0)~\wrt x  \leq \int_0^1 \phi_T(x+iy, 0) \chi^+(x)~\wrt x  \\ & \leq \int_0^1  \varphi_T(x+iy, 0)\chi^+(x)~\wrt x + O\left(\frac{1}{T^{3/2}}\right) \left(\sum_{m \in \ZZ \backslash \{0\}}\frac{\left|\widehat{\frac{\wrt{}^2\chi^+}{\wrt x^2}}(m)\right|}{4\pi^2m^2} + \int_0^1 \chi^+(x)~\wrt x \right)
\\ & \leq \int_\alpha^\beta  \varphi_T(x+iy, 0)~\wrt x +2 \varepsilon M + O\left(\frac{1}{\varepsilon^2T^{3/2}}\right)
\\ & \leq \int_\alpha^\beta  \varphi_T(x+iy, 0)~\wrt x + O\left(\frac{1}{T^{\delta}}\right)\end{align*} where the implied constant depends on $h$, $\kappa_1$, and $\kappa_j$. The second inequality follows an application of Corollary~\ref{coro:MovingToFixed} and that $\chi^+$ is smooth, and the third inequality comes from Lemma~\ref{lemm:bndFourierSmoothCutoff}.  The analogous proof using $\chi^-$ in place of $\chi^+$ allows us to obtain the desired lower bound.  Consequently, we have that \begin{align*} \int_\alpha^\beta \phi_T &(x+iy, 0)~\wrt x = \int_\alpha^\beta  \varphi_T(x+iy, 0)~\wrt x + O\left(\frac{1}{T^{\delta}}\right)
  \end{align*} where the implied constant depends on $h$, $\kappa_1$, and $\kappa_j$.

Since $h + M \geq 0$, applying the above proof with $h$ replaced by $h +M$ yields the desired result for $h+M$.  Applying the above proof with $h$ replaced by $M$ yields the desired result for $M$.  Subtracting these two results yields the desired result for general $h$ and proves the desired result in all cases.  \end{proof}

\end{document}